\documentclass[reqno,10pt]{amsart}
\usepackage{amssymb,amsmath,amsthm,enumerate,amscd,graphicx,color}
\usepackage[usenames,dvipsnames,svgnames,table]{xcolor}
\usepackage{enumitem}
\usepackage{pdfsync}
\usepackage{float}
 \usepackage[colorlinks=true, pdfstartview=FitV, linkcolor=blue, citecolor=blue, urlcolor=blue]{hyperref}
\theoremstyle{rem}
\numberwithin{equation}{section}

\newtheorem{theorem}{Theorem}[section]
\newtheorem{corollary}{Corollary}[section]
\newtheorem{lemma}{Lemma}[section]%
\newtheorem{proposition}{Proposition}[section]
\theoremstyle{definition}
\theoremstyle{remark}
\newtheorem{definition}{Definition}[section]
\newtheorem{remark}{Remark}[section]
\newtheorem{example}{Example}[section]

\newcommand{\B}{\mathcal{B} }
\newcommand{\bC}{{\mathbb C}}
\newcommand{\bZ}{{\mathbb Z}}
\newcommand{\bN}{{\mathbb N}}

\newcommand{\End}{\operatorname{End}}

\begin{document}

\author{Victor G.\,Kac}
\address{Department of Mathematics,
Massachusetts Institute of Technology, Cambridge, Massachusetts 02139, USA}
\email{kac@math.mit.edu}
\author{Natasha Rozhkovskaya}
\address{Department of Mathematics, Kansas State University, Manhattan, KS 66502, USA}
\email{rozhkovs@math.ksu.edu}
\author{Johan van de Leur}
\address{Mathematical Institute, Utrecht University, P.O. Box 80010, 3508 TA-Utrecht,
The Netherlands}
\email{J.W.vandeLeur@uu.nl}
\thanks{}

\begin{abstract}
We show that any polynomial tau-function of the $s$-component KP and the BKP hierarchies can be  interpreted as 
a zero mode of an appropriate combinatorial generating function. As an application, we obtain explicit formulas
for all polynomial tau-functions of these hierarchies in terms of Schur polynomials and $Q$-Schur  polynomials respectively. 
We also obtain formulas for  polynomial tau-functions  of the reductions of  the $s$-component KP hierarchy associated to partitions in $s$
parts.
\end{abstract}
\title {Polynomial tau-functions of the KP,  BKP, and  the $s$-component KP  hierarchies}

\maketitle
\section{Introduction}

In \cite{Sato}  M.\,Sato introduced the KP  hierarchy of evolution equations, and his ideas were  further developed by the Kyoto school  \cite{DJKM4} 
-- \cite{JM}. In particular, related hierarchies were  introduced and studied  in these papers, including  the BKP and the DKP hierarchy  \cite{DJKM2},  \cite{JM}, the modified KP-hierarchy \cite{JM}, the $s$-component KP hierarchy \cite{DJKM4}. 
M.\,Sato expressed  in  \cite{Sato}    solutions  of the KP hierarchy  through  tau-functions, and   showed   that the set of  polynomial tau-functions  form an infinite Grassmann manifold and contains the Schur  polynomials, see \cite{Kac-bomb} (resp. \cite{KL-sKP}) for an exposition  of the theory  of  tau-functions for the KP (resp. $s$-component KP) hierarchy.
  In \cite{You1}, \cite{You2} Y.\,You  proved that polynomial tau-functions   of the BKP, DKP and MDKP hierarchies include   Schur  $Q$-polynomials.    

In the recent series of papers  \cite{KL-KP}, \cite{KL-BKP}, \cite{KL-sKP}  the first and the third authors described  all polynomial tau-functions of the KP, BKP, DKP, MDKP, $s$-component KP,   and the modified KP hierarchies. In particular, they showed in   \cite{KL-KP}, \cite{KL-sKP} that any polynomial tau-function of the KP-hierarchy is obtained from a Schur polynomial by  certain shifts of arguments, and in \cite{KL-BKP}  they proved that  any polynomial tau-function of the BKP or the DKP hierarchy is  described by a Pfaffian formula that can be  also related to Schur $Q$-polynomials by certain shifts of arguments.  In the recent paper \cite{R-Q}  the second author showed that the modes of
  certain generating series, involving Schur Q-polynomials, are
  tau-functions of the BKP hierarchy.

In this note we show that any polynomial   tau-function  of the KP, the BKP and  the $s$-component KP  hierarchy  can be interpreted as a zero-mode of an appropriate combinatorial  generating function.
We give explicit formulae  for these generating functions, which have determinant type in the KP and $s$-component KP cases, and Pfaffian type in the BKP case.
 As a result  we recover the formulas for all polynomial tau-functions of  the KP and the BKP hierarchies  obtained in  \cite {KL-KP}  and \cite{KL-BKP}. Moreover, we find a new beautiful formula for all polynomial tau-functions of the  $s$-component KP  hierarchy that generalizes the one for $s=1$ from \cite{KL-KP}, and its reductions associated to  arbitrary partitions in $s$ parts  introduced  in \cite{KL-2003}.

The main instrument of our proofs is  the identification of   the  polynomials of the boson space  with symmetric functions. Using the boson-fermion correspondence and the well-known properties of symmetric functions, we  consider the explicit description of   the actions of charged free fermions and neutral fermions  on  these spaces of symmetric functions. Application of these quantum fields to the vacuum vector produces the generating functions for polynomial   tau-functions, which are identified with  solutions of the  corresponding bilinear identities. 

In Section \ref{Sec1}  we review necessary facts on symmetric functions. We describe 
polynomial tau-functions for   the  KP  case    in Section \ref{Sec2}, the  BKP case   in Section \ref{Sec3},  and the $s$-component KP case  in Section \ref{Sec4}.
We treat the polynomial tau-functions of the reductions of the  $s$-component KP hierarchy associated to partitions in $s$
parts in  Section \ref{SecKdV}.

Throughout the paper  $\mathbb{N}$  (resp. $\bZ_+$) denotes the set of positive (resp. non-negative) integers.

\subsection*{Acknowledgments} 
We thank the referee for suggesting some improvements.

\section{Symmetric functions and formal distributions: overview }\label{Sec1}

\subsection{Symmetric functions}
Recall the definition and properties of symmetric functions \cite{Md}, \cite{Stan}.
Let ${\bf x}=\{x_1,x_2, x_3\dots\}$ be an infinite set of indeterminates. 
 Consider the algebra of formal power series
 $\bC[[{\bf x}]]=\bC[[x_1,x_2, x_3,\dots]]$.
 By a {\it symmetric function} we mean an element of this algebra, invariant with respect to all permutations of the indeterminates, which lies in the span of monomial  symmetric functions. 
Let $\lambda=(\lambda_1\ge \dots\ge \lambda_l\ge 0)$ be a partition. The {\it monomial symmetric function}  corresponding to $\lambda$
is 
\[
m_\lambda=\sum_{(i_1,\dots,i_l)\in \mathbb N^l} x^{\lambda_1}_{i_1}\dots x^{\lambda_l}_{i_l}.
\]
Then the  vector space spanned by all  { monomial symmetric functions} is a subalgebra $\Lambda$ of $\bC[[{\bf x}]]$,  called the  subalgebra of symmetric functions.

For a partition $\lambda=(\lambda_1\ge \dots\ge \lambda_l\ge 0)$, the  {\it Schur symmetric   function} $s_\lambda$ is defined as 
\[
s_\lambda(x_1,x_2,\dots ) =\sum_{T} {\bf x}^{T},
\]
where the sum is over all semistandard tableaux of shape $\lambda$ (see e.g.  \cite{Stan} for details).
These functions form a   basis of the vector space $\Lambda$.

 {\it Complete symmetric functions}  $h_k= s_{(k)}$ are given by the formula

\[
h_k(x_1,x_2\dots)=\sum_{1\le i_1\le \dots \le i_k<\infty} x_{i_1}\dots x_{i_k},
\]
while {\it elementary symmetric functions} $e_k= s_{(1^k)}$  by
 \[e_k(x_1,x_2\dots)=\sum_{1\le i_1< \dots < i_k<\infty} x_{i_1}\dots x_{i_k}.
\]
 {\it Power sums} $p_k$  are symmetric functions defined by 
\[
p_k(x_1,x_2,\dots)=\sum_{i\in \bN} x_i^k.
\]
It will be convenient  to set 
$h_{-k}(x_1,x_2\dots)=e_{-k}(x_1,x_2\dots)=p_{-k}(x_1,x_2\dots)=0
$
 for 
$k>0$ and $h_0=e_0=p_0=1$.
 Recall that each of these  three families generate  $\Lambda$ as the algebra of polynomials
\[
 \Lambda=\bC[h_1, h_2,\dots]=\bC[e_1, e_2,\dots]=\bC[p_1,p_2,\dots].
\]

For every symmetric function $f\in \Lambda$  we can {\it  associate} a polynomial $F(t_1, t_2,\dots )$ by expressing $f$ as a polynomial in the $p_j$'s  and  replacing   $p_j$  by $jt_j$ .

Introduce  the polynomials $S_k(t_1, t_2,\dots)$, $k\in \bZ$, by the generating series
 \[
\sum_{k\in \bZ}S_k(t_1,t_2,\dots, )u^k=exp\left( \sum_{j\in \bN} t_j u^j\right).
\]
For a partition $\lambda$  having $l$ parts, the corresponding  Schur polynomial $S_\lambda(t_1, t_2,\dots)$
is defined by the Jacobi\,--\,Trudi  formula 
\begin{align*}
S_\lambda(t_1, t_2,\dots) =\det [S_{\lambda_i+j-i}(t_1, t_2, \dots )]_{1\le i, j\le l}.
\end{align*}
This polynomial is associated to  the symmetric function $s_\lambda$  as described above (\cite{Md} I.2 Example 8 and  I.3 (3.4)).

 We introduce on $\Lambda$ a natural  scalar product, where  the set of Schur symmetric functions $\{s_\lambda\}$, labeled by partitions $\lambda$, form an orthonormal basis:
$
 	<s_\lambda,s_\mu >=\delta_{\lambda, \mu}.
$
Then for any  linear operator  acting on the  vector space $\Lambda$ one can define the corresponding adjoint operator. In particular,  any symmetric function  $f\in \Lambda $  defines an operator of multiplication
$
f:g\mapsto fg$  for any $g\in \Lambda$,
and the corresponding   adjoint operator $f^\perp$  is defined by the standard rule
$<f^\perp g_1, g_2>=<g_1,fg_2>$ for all   $g_1,g_2\in \Lambda$.

It is known  (\cite{Md},  I.5  Example 3) that
$
p_n^\perp= n\frac{\partial}{\partial p_n}.
$
Since any element  $f\in \Lambda$  can be expressed as a polynomial  function of power sums 
\begin{align*}
f= F(p_1,p_2, p_3, \dots),
\end{align*}
the corresponding adjoint operator $f^\perp$ is a polynomial differential operator with constant coefficients
\begin{align*}
f^\perp= F(\partial/\partial p_1,2\partial/\partial p_2, 3\partial/\partial p_3,\dots).
\end{align*}
In particular, $e_k$ and $h_k$ are  homogeneous polynomials of degree $k$ in  $(p_1,p_2, p_3,\dots )$, so
 the adjoint operators $e^\perp_k$ and $h^\perp_k$ are  homogeneous polynomials of degree $k$ in $(\partial/\partial p_1,2\partial/\partial p_2, \dots)$, which implies
  the following statement. 
\begin{lemma}\label{eh_dp}
For any symmetric function $f\in \Lambda$   there exists  a positive  integer   $N= N(f)$, such that
 \[e^\perp_l(f)=0\quad \text{and}\quad h^\perp_l(f)=0\quad \text{ for all\quad   $l\ge N$}.
 \] 
 \end{lemma}
 
 \subsection{Digression on formal distributions and quantum fields \cite{Kac-begin}} 
 Given  a vector space $M$, an {\it $M$-valued formal distribution} is a biletaral series 
in the  indeterminate $u$ with coefficients  in $M$:
\[
a(u)=\sum_{n\in \bZ}a_n u^n,\quad a_n\in M.
\]
We  denote as $M[[u,u^{-1}]]$ the vector space of  all $M$-valued formal distributions.
We also  use the notation  $M[u]$ for the space  of polynomials,  $M[[u]]$ for the space of power series, $M[u, u^{-1}]$ for the space of Laurent polynomials, and $M((u))$   for the space of formal Laurent series.

A formal distribution in two indeterminates $u$ and $v$ is defined similarly. The most famous example is the formal delta-function  $\delta(u,v)$, which is the $\bC$-valued  formal 
distribution  in  variables  $u$ and $v$, characterized by the condition 
\[
\text{Res}_u a(u)\delta(u,v)=a(v)
\]
for any $M$-valued formal distribution $a(u)$,  where $\text{Res}_u$ denotes the coefficient of $u^{-1}$. The  following  is an explicit formula: 
\[
\delta(u,v)= \sum_{\substack{i,j\in \bZ\\ i+j=-1 }} {u^i}{v^{j}} = i_{u,v}\left(\frac{1}{u-v}
\right) -i_{v,u} \left(\frac{1}{u-v}
\right),
\] 
where 
$i_{u,v}$ (resp. $i_{v,u}$) denotes  the expansion in geometric series  in the  domain  $|u|>|v|$ (resp. $|u|>|v|$).
The main property of the formal delta-function is that for any $M$-valued 
formal distribution $a(u)$ one has
\begin{align}\label{23a}
 a(u)\delta(u,v)=a(v)\delta(u,v).
\end{align}

A special case of a formal distribution  is a {\it quantum field}, which is an $\End\, M$-valued formal distribution  $a(u)$, such that
$a(u)m\in M((u))$ for all $m\in M$. 
The following lemma is obvious.
\begin{lemma}\label{qfl}
If $a(u)\in (\End M)\,[[u]]$, and $b(u)$ is a quantum filed, then $a(u)b(u)$ is a quantum field.
\end{lemma}

\subsection{Generating series  of polynomial  differential operators acting  on $\Lambda$} 
\label{secGKP}
Denote by $\mathcal D$  the algebra  of differential operators  acting on $\Lambda=\bC[p_1,p_2,\dots]$, which consists of finite sums  
\[
\sum_{i_1,\dots, i_m} F_{i_1, \dots i_m}(p_1, p_2,\dots) {\partial _{p_1}^{i_1}}\dots  {\partial _{p_m}^{i_m}},
\] 
where coefficients  $F_{i_1, \dots i_m}(p_1, p_2,\dots)$ are  polynomials in $(p_1, p_2,\dots)$.
Then operators of multiplication $p_n, h_n, e_n$, their adjoints $p_n^\perp, h_n^\perp, e_n^\perp$ along with  their products  are elements of   $\mathcal D$.

Consider the generating series  of  complete and  elementary symmetric functions
\begin{align}\label{HE}
H(u)=\sum_{k\in \bZ_+}  {h_k}{u^k}=\prod_{i\in \bN} \frac{1}{1-x_iu},\quad \quad 
E(u)=\sum_{k\in \bZ_+}{ e_k}{ u^k}=\prod_{i\in \bN} {(1+x_iu)},
\end{align}
which are elements of  $ \Lambda[[u]]$. We will use the same notation for  the generating series of the corresponding multiplication operators
$H(u),E(u)\in \mathcal D[[u]]$. Similarly, we define $E^\perp(u), H^\perp(u)\in \mathcal D[[u^{-1}]]$ as
\begin{align}\label{perpHE}
  E^\perp(u)= \sum_{k\in \bZ_+}\frac {e^\perp_k} {u^k},\quad H^\perp(u)= \sum_{k\in \bZ_+} \frac{h^\perp_k} {u^k}.
\end{align}

The following properties of these   generating series with  coefficients in $\mathcal D$ are well known  (e.g. \cite {Md}, I.5). 
\begin{proposition}\label{prop_rel} We have in  $\mathcal D[[u]]$ (resp. in  $\mathcal D[[ u^{-1}]]$\,) 
\begin{align}\label{HE1}
H(u) E(-u)=1,
\quad 
H^\perp(u) E^{\perp}(-u)=1,
\end{align}
\begin{align}\label{HEP} 
H(u)&= exp\left(\sum_{n\in \bN} \frac{p_n}{n}{u^n}\right),\quad 
E(u)= exp\left(-\sum_{n\in \bN} \frac{(-1)^{n} p_n}{n}{u^n}\right),
\end{align}
\begin{align}\label{DHEP}
  E^\perp(u)= exp\left(-\sum_{k\in \bN} {(-1)^k} \frac{\partial}{\partial p_k} \frac {1}{u^k}\right),
\quad     H^\perp(u)= exp\left( \sum_{k\in \bN}\frac{\partial}{\partial p_k} \frac{1}{u^k}\right),
\end{align}
\end{proposition}

The next set  of commutation  relations follows from the  definitions of  the generating series above (\cite {Md}, I.5 Example 29). These are the main identities that we use  to prove other commutation relations later in this note. 
\begin{lemma} \label{propHE} We have  the following equations in   $\mathcal D[[u^{-1}, v]]$.

  \begin{align*}
\left(
1-\frac{v}{u}
\right)E^\perp(u)E(v)= E(v)E^\perp(u),\\
\left(
1-\frac{v}{u}
\right)H^\perp(u)H(v)= H(v)H^\perp(u),\\
H^\perp(u)E(v)= \left(
1+\frac{v}{u}
\right)E(v)H^\perp(u),\\
E^\perp(u)H(v)= \left(
1+\frac{v}{u}
\right)H(v)E^\perp(u).
\end{align*}
\end{lemma}

\subsection{Schur symmetric $Q$-functions and Schur $Q$-polynomials}
 Introduce  one more family of symmetric functions $\{q_k(x_1,x_2,\dots)\}_{k\in \bZ}$ as the  coefficients of the expansion of $Q(u)\in  \Lambda[[u]]$, where
 \begin{align}\label{shurq}
Q(u) =\sum_{k\in \bZ} q_k u^k= E(u) H(u).
\end{align}
Note that $q_k=\sum_{i=0}^k e_ih_{k-i}$ for $k>0$, $q_0=1$, and $q_k=0$ for $q<0$.
For $a,b\in \bZ_{+}$  let
\[
q_{a,b}=q_aq_b+2\sum_{i\in \bZ} (-1)^i q_{a+i} q_{b-i}.
\]
Then  (\cite{Md},  III.8)
\begin{align}\label{qab_qba}
q_{a,b}=-
q_{b,a}.
\end{align}

\begin{proposition}\label{prop_rel2}  (\cite{Md},  III.8) We have in  $\mathcal D[[u]]$ (resp. in  $\mathcal D[[ u^{-1}]]$\,) 
  \begin{align}\label{QS}
Q(u) =S(u)^2, \quad\text{where}\quad   
S(u)=exp\left(\sum_{n\in \bN_{odd}}\frac{p_{n}}{n}{u^{n}}\right),
\end{align}
\begin{align*}
S^{\perp}(u)=exp\left(\sum_{n\in \bN_{odd}} \frac{\partial}{\partial p_n}\frac{1}{u^{n}}\right).
\end{align*}
Here $\bN_{odd}=\{1,3,5,\dots\}$.
\end{proposition}
Let $Q_k(t_1, t_3, \dots)$ be   the polynomial associated to   the symmetric function $q_k$ as described above. 
By  (\ref{QS}) we have: 
\begin{align*}
Q_k(t_1, t_3, \dots)= S_k(2t_1, 0, 2t_3, 0,\dots). 
\end{align*}

Recall that the {\it  Pfaffian}  of a skew-symmetric  matrix $M=[M_{ij}]$ of size $2l\times 2l$ is defined as
$$
\mathrm{Pf}\,[M]=\sum_{\sigma\in S'_{2l}}  sgn(\sigma) M_{\sigma(1) \sigma(2)}  \cdots M_{\sigma(2l-1) \sigma(2l)},
$$
where $S'_{2l}$  is the subset of the permutation group $S_{2l}$ that consists of $\sigma\in S_{2l}$  such that $\sigma(2k-1)<\sigma(2k)$ for $1\le k\le l$ and
$\sigma(2k-1)<\sigma(2k+1)$ for $1\le k\le l-1$. 

If $\lambda=(\lambda_1,\dots, \lambda_{2m})$ is a strict partition, i.e. $\lambda_1>\dots> \lambda_{2m}\ge 0$, then the matrix 
$M_\lambda= (q_{\lambda_i,\lambda_j})$ is skew-symmetric by (\ref{qab_qba}), and the {\it Schur symmetric  $Q$-function} $q_\lambda$  is defined as
\[
q_\lambda(x_1, x_2, \dots)=\text{Pf} \,M_\lambda.
\]

For the strict partition  $\lambda=(\lambda_1,\dots, \lambda_{2m})$ introduce   the {\it  extended  Schur  $Q$-polynomial} $\tilde Q_\lambda(t)$ as follows:
\begin{align*}
 \tilde Q_a(t)&= S_a(2t) \quad \text{for}\quad a\in \bZ_+,\\
\tilde Q_{a,b}(t)&=\tilde Q_a(t)\tilde Q_b(t)+2\sum_{i\in \bZ} (-1)^i \tilde Q_{a+i}(t) \tilde Q_{a-i}(t)   \quad \text{for}\quad a>b, \quad a,b\in \bZ_+,\\
\tilde Q_{a,b}(t)&=-\tilde Q_{b,a}(t) \quad \text{for}\quad a<b, \quad a,b\in \bZ_+, \quad \text{and}\quad Q_{a,a}(t)=0 \quad \text{for} \quad a\in \bZ_+,\\
\tilde Q_\lambda(t)&= \text{Pf}\, [\tilde Q_{\lambda_i, \lambda_j}(t)] .
\end{align*}
Then 
$
Q_\lambda (t_1, t_3, \dots)=\tilde Q_\lambda(t_1,0, t_3,0\dots)
$ is  the {\it Schur $Q$-polynomial} associated to  the Schur symmetric $Q$-function $q_\lambda$.
.
\begin{remark}
One can define a Schur $Q$-polynomial also for a strict
partition with odd number of parts,
for $\lambda = (\lambda_1, \lambda_2,\ldots, \lambda_{2m+1})$, define (see  \cite{Jing1}):
\[
Q_\lambda=  Q_{\lambda_1}Q_{\lambda_2,\lambda_3, \ldots ,\lambda_{2m+1}}-
Q_{\lambda_2}Q_{\lambda_1,\lambda_3, \ldots ,\lambda_{2m+1}}+\ldots+Q_{\lambda_{2m+1}}Q_{\lambda_1,\lambda_2, \ldots, \lambda_{2m}}\, .
\]
\end{remark}
\section{Polynomial  tau-functions of the  KP hierarchy}\label{Sec2}
The goal of this  section  is to  describe  polynomial  tau-functions  of the KP hierarchy. 
 We use the bosonic formulation of   the bilinear KP  identity and write the action of the Clifford algebra of charged free fermions  on the boson Fock space through 
symmetric functions. Using  commutation relations  of  Lemma \ref{propHE} we prove that all polynomial tau-functions  of the  KP hierarchy can be obtained as coefficients of certain generating functions. 
For more information on  formal distributions, quantum fields,  Clifford algebra of fermions, and the boson-fermion correspondence we refer to 
\cite{Kac-begin}, \cite{Kac-bomb}.
\subsection{Fermionic fields on the boson Fock space} \label{3.1}

Let $\Lambda=\bC[p_1,p_2,\dots]$, and  
consider the  {\it boson Fock space} $\B= \bC[ z, z^{-1}]\otimes \Lambda$. Note that we have the {\it charge} decomposition 
\[\B=\oplus_{m\in \bZ} \B^{(m)}, \quad \text{
where }\quad 
\B^{(m)}=  z^m\,\Lambda.
\]
Let $R(u)$ act  on the  elements of the  form  $z^m f$,  $f \in \Lambda$,  $m\in \bZ$, by the rule
\[
R(u) (z^mf)={ z^{m+1}}{u}^{m+1} f,
\]
then  $R^{-1}(u)$ acts as
\[
R^{-1}(u) (z^mf)={ z}^{m-1} u^{-m}f.
\]
Define the formal distributions $\psi^\pm(u)$  of operators  acting on the  space $\B$  through the action of  $R^{\pm 1}(u)$ and the $\mathcal D$-valued generating series (\ref{HE}), (\ref{perpHE}):
\begin{align}
\psi^+(u)&=u^{-1}R(u)H(u) E^{\perp}(-u),\label{deff1}\\
\psi^-(u) &=R^{-1}(u)E(-u) H^{\perp}(u),\label{deff2}
\end{align}
or, in  other words, for any $m\in \bZ$ and any $f\in \Lambda$,
\begin{align}
\psi^+(u)(z^mf)&=z^{m+1}u^{m}H(u) E^{\perp}(-u)(f),\label{psi+}\\
\psi^-(u)(z^mf)&=z^{m-1}{u^{-m}}E(-u) H^{\perp}(u)(f). \label{psi-}
\end{align}
Let the operators   $\{\psi^{\pm}_{i}\}_{i\in \bZ+1/2}$ be the coefficients of the expansions
 \[
\psi ^\pm(u)= \sum_{i\in \bZ+1/2}\psi^{\pm}_i u^{-i-1/2}.
\] 
Obviously, the operators $\psi ^\pm_i$ change the charge by $\pm 1$. These operators are called the {\it  charged free fermions}.
\begin{proposition}\label{fermR}
 \begin{enumerate}[label=\alph*)]
\item
Formulae (\ref{psi+}) and (\ref{psi-}) define quantum fields $\psi^\pm(u)$ of  operators acting on  the space  $\B$.
\item
The following   relations of quantum fields hold:
\begin{align}
\psi^\pm(u)\psi^\pm(v)+ \psi^\pm(v)\psi^\pm(u)&=0,\label{psi_rel1}
\\
\psi^+(u)\psi^-(v)+ \psi^-(v)\psi^+(u)&=\delta(u,v).\label{psi_rel2}
\end{align}
Equivalently,   (\ref{psi+}) and  (\ref{psi-})  define the action of the Clifford algebra  of charged free fermions  on the  space  $\B$:
\begin{align}\label{ckl}
\psi_k^\pm\psi_l^\pm& +\psi_l^\pm\psi_k^\pm=0,\\
\psi_k^+\psi_l^- &+\psi_l^-\psi_k^+=\delta_{k, -l}, \quad k,l\in \bZ+1/2.
\end{align}
\item 
$z \psi^\pm_n=\psi^\pm_{n\mp 1} z, \quad n\in \bZ+1/2$.
\end{enumerate}

\end{proposition}
\begin{proof} a) Follows from  and  Lemma  \ref{eh_dp} and  Lemma \ref{qfl}.

b)
By Lemma \ref{propHE},  for  any $m\in \bZ$ and $f\in \Lambda$ we can write
\begin{align*}
\psi^+(u)\psi^+(v) (z^m f)&=z^{m+2} u^{m+1}v^{m} H(u) E^{\perp}(-u)H(v) E^{\perp}(-v) (f)\\
&=z^{m+2} u^{m+1}v^{m} \left(1-\frac{v}{u}
\right)H(u)H(v) E^{\perp}(-u) E^{\perp}(-v)(f)\\
&=z^{m+2} u^mv^{m} \left(u-{v}
\right)H(u)H(v) E^{\perp}(-u) E^{\perp}(-v)(f),
\end{align*}
and, similarly, 
\begin{align*}
\psi^-(u)\psi^-(v)(z^m f)&=z^{m-2}u^{-m+1}v^{-m}E(-u) H^{\perp}(u)E(-v) H^{\perp}(v)(f)\\
&=z^{m-2}u^{-m+1}v^{-m}\left(1-\frac{v}{u}
\right)E(-u)E(-v) H^{\perp}(u) H^{\perp}(v)(f) \\
&=z^{m-2}u^{-m}v^{-m}\left(u-{v}
\right)E(-u)E(-v) H^{\perp}(u) H^{\perp}(v) (f).
\end{align*}
Switching the  roles of $u$ and $v$  and adding the corresponding products of quantum fields we get (\ref{psi_rel1}).

 Using the notation 
$i_{u,v}\,F(u,v)$ for the expansion of a rational function $F(u,v)$ in the region $|u|>|v|$, write
\begin{align*}
i_{u,v}\left(
1-\frac{v}{u}
\right)^{-1}=1+ \frac{v}{u} +\frac{v^2}{u^2} +\frac{v^3}{u^3} +\dots\quad .
\end{align*}
 It  is  an inverse of  $\left(
1-\frac{v}{u}
\right)$ in the  algebra $\mathcal D[[u^{-1},v]]$.
Hence we get   from  Lemma \ref{propHE} that
in $\mathcal D[[u^{-1},v]]$
\begin{align*}
E^\perp(-u)E(-v)= i_{u,v}\left(
1-\frac{v}{u}
\right)^{-1}E(-v)E^\perp(-u),
\end{align*}
and  similarly
in $\mathcal D[[u,v^{-1}]]$
\begin{align*}
H^\perp(v)H(u)= i_{v,u}\left(
1-\frac{u}{v}
\right)^{-1}H(u)H^\perp(v).
\end{align*}
For the proof  of  (\ref{psi_rel2}), using  these relations and claim a),
 we  can write for any $f\in \Lambda$ and any $m\in \bZ$,
\begin{align*}
\psi^+(u)\psi^-(v) (z^m f)
=z^mu^{m-1}v^{-m}H(u) E^{\perp}(-u)E(-v)  H^{\perp}(v)(f),
\\
=  i_{u,v}\left(1-\frac{v}{u}
\right)^{-1}z^mu^{m-1}v^{-m}H(u)E(-v) E^{\perp}(-u) H^{\perp}(v)( f),
\end{align*}
  and similarly, 
\begin{align*}
\psi^-(v)\psi^+(u)(z^m f)
=i_{v,u} \left(1-\frac{u}{v}
\right)^{-1}z^mu^{m}v^{-m-1}H(u)E(-v) E^{\perp}(-u) H^{\perp}(v)(f).
\end{align*}

Then for the sum of  products of quantum fields  $\psi^+(u)\psi^-(v)+\psi^-(v)\psi^+(u)\in\mathcal D[[ u, v, u^{-1},v^{-1}]]$ we can write
\begin{align*}
&(\psi^+(u)\psi^-(v)+\psi^-(v)\psi^+(u))(z^mf)\\
&
=\left(\frac{1}{u} i_{u,v}\left(1-\frac{v}{u}
\right)^{-1} +\frac{1}{v}i_{v,u} \left(1-\frac{u}{v}
\right)^{-1}
\right)
\frac{z^mu^m}{v^{m}}H(u)E(-v) E^{\perp}(-u) H^{\perp}(v) (f) .
\end{align*}
Note that 
\begin{align*}
\frac{1}{u} i_{u,v}\left(1-\frac{v}{u}
\right)^{-1} +\frac{1}{v}i_{v,u} \left(1-\frac{u}{v}
\right)^{-1}
=i_{u,v}\left(\frac{1}{u-v}
\right) +i_{v,u} \left(\frac{1}{v-u}
\right)=\delta(u, v).
\end{align*}
Hence
\begin{align*}
\psi^+(u)\psi^-(v)+\psi^-(v)\psi^+(u)(z^m f)
=z^m\frac{u^m}{v^{m}} \,\delta(u, v)H(u)E(-v) E^{\perp}(-u) H^{\perp}(v) (f) .
\end{align*}
 Using (\ref{23a})  and Proposition \ref{prop_rel} we get 
\[
\delta(u, v) \frac{u^m}{v^{m}}H(u)E(-v) E^{\perp}(-u) H^{\perp}(v)=\delta(u, v) H(u)E(-u) E^{\perp}(-u) H^{\perp}(u)= \delta(u, v).
\]
Hence
\begin{align*}
&(\psi^+(u)\psi^-(v)+\psi^-(v)\psi^+(u))(z^m f)= z^m\delta(u, v) f,
\end{align*}
and  (\ref{psi_rel2}) is proved. 

c) is straightforward from definitions.
 \end{proof}

\begin{remark}\label{rema}
From  (\ref{HEP}), (\ref{DHEP}) one immediately  gets the bosonic form of   
 $\psi^{\pm}(u)$ :
\begin{align*}
 \psi^+(u)&=
 u^{-1}R(u)
 \exp \left(\sum_{n\ge 1}\frac{p_n}{n}{u^n}  \right) \exp \left(-\sum_{n\ge 1}\frac{\partial} {\partial p_n}\frac{1} {u^{n}}\right),
\\
\psi^-(u)&=
R^{-1}(u)
\exp \left(-\sum_{n\ge 1}\frac{p_n}{n} {u^n}  \right) \exp \left(\sum_{n\ge 1}\frac{\partial} {\partial p_n} \frac{1}{u^{n}}\right).
\end{align*}
It follows that 
$\psi^+_{k-m}(z^m)=0$  and $\psi^-_{k+m}(z^m)=0$ if $k>0$.
\end{remark}

\subsection{The bilinear  KP identity} \label{subKPbil}
 Recall  
  \cite{DJKM3}, \cite{DJKM2}, \cite{Kac-bomb} that the {\it  bilinear KP identity} is the equation of the form
\begin{align}\label{binKP}
\Omega\,  (\tau \otimes \tau)=0
\end{align}
on a function $\tau=z^m\tau(p_1,p_2,\dots)$ from the formal completion of the space $\B^{(m)}=z^m\Lambda$,
where 
\[
 \Omega= \sum _{k\in\bZ+1/2}\psi^+ _k\otimes \psi^-_{-k}.
\]
Note that this equation is independent on $m\in \bZ$, 
By Proposition \ref{fermR}\,c)   $z\tau$ is a solution of the bilinear 
equation (\ref{binKP}) iff $\tau$ is. Hence the solutions of (\ref{binKP}) in 
$\B^{(m)}$ are obtained from those in $\B^{(n)}$ by multiplying by $z^{m-n}$.

Non-zero solutions of  (\ref{binKP}) are called {\it tau-functions of the KP hierarchy}.  Accordingly, we will say that  a non-zero
solution of (\ref{binKP}) is a polynomial  tau-function,   if it is  a polynomial function in the variables  $(p_1,p_2,\dots)$ times $z^m$ (hence, an element of $\B ^{(m)}$ rather than  a completion of the space $\B^{(m)}$).
By Remark \ref{rema}, the  vector $z^m$ is obviously  a solution of (\ref{binKP}). With the help of the commuting with $\Omega$ operators   one can construct more examples of tau-functions.

\begin{lemma} \label{commuteKPlemma} Let  
  $X=\sum_{i> M} A_i \psi^+_{i}$, where $A_i\in \bC$,  $M\in \bZ$. Then
$X^2=0$.
\end{lemma}
\begin{proof}
Note that  by Proposition \ref{fermR} (a) 
for any $f\in \Lambda$  and   $m\in \bZ$ there exists $N$, such that
$X(z^mf)=\sum_{i=M+1/2}^{N+1/2} A_i \psi^+_{i}(z^mf)$,  hence $X$ and $X^2$    are well-defined operators on $\B$.
Commutation relations (\ref{ckl}) complete the proof. 
\end{proof}
\begin{lemma}
\label{OmXKPlemma}
 Let 
  $X=\sum_{i> M} A_i \psi^+_{i}$, where $A_i\in \bC$, $M\in \bZ$. Then
\begin{align*}
		\Omega (X\otimes X)= (X\otimes X)\Omega. 
	\end{align*}
\end{lemma}
\begin{proof}
Note that $\psi^-_{-k} X= -X\psi^-_{-k} + A_k$.
Then 
\begin{align*}
\Omega (X\otimes X)&=\sum_{k\in \bZ +1/2} \psi^+_k X\otimes \psi^{-}_{-k} X=\sum_{k\in \bZ+1/2} (-X\psi^+_k)\otimes  (-X\psi^{-}_{-k} + A_k)\\
&= (X\otimes X)\Omega - X\sum_{k\in \bZ+1/2} A_k\psi ^+_k \otimes 1
= (X\otimes X)\Omega - X^2 \otimes 1=(X\otimes X)\Omega.
\end{align*}
\end{proof}
 \begin{corollary}\label{corKP}
 Let $\tau\in \B^{(m)}$  be a  tau-function of
the KP  hierarchy.
  Let
  $X=\sum_{i> M}^{\infty} A_i \psi^+_{i}$, where $A_i\in \bC$,  $M\in \bZ$.
Then   $\tau^\prime =X\tau\in \B^{(m+1)}$ is also a tau-function of  the KP  hierarchy.
 \end{corollary}
\subsection{Generating function for  polynomial tau-functions of the KP hierarchy}\label{sec2.3} 
Let 
\begin{align}\label{QKP}
G(u_1,\dots, u_l)=\prod_{1\le i<j\le l}\left(u_j-{u_i}\right)\prod_{i=1}^{l} H(u_i).
\end{align}
be a  formal power series  in $u_1, \dots, u_l$ with coefficients in the algebra $\Lambda$.
From Lemma  \ref {propHE}  we have
\begin{align*}
\psi^+(u_l)\dots \psi^+(u_1) \, (z^k f)=
z^{k+l} u_l^{k+l-1}\,\dots\, u_1^{k} H(u_l) E^{\perp}(-u_l)\dots H(u_1) E^{\perp}(-u_1) (f)\\
=z^{k+l} u_l^{k+l-1}\,\dots\, u_1^{k}\prod_{1\le i<j\le l}\left(1-\frac{u_i}{u_j}\right) H(u_l) \dots H(u_1) E^{\perp}(-u_l)\dots E^{\perp}(-u_1) (f).
\end{align*}
In particular, 
\begin{align}\label{fermQ}
\psi^+(u_l)\dots \psi^+(u_1) \, (z^k)=  z^{k+l} u_l^k\dots u_1^kG(u_1,.., u_l).
\end{align}
 
In  \cite{KL-KP}, \cite{KL-sKP}, polynomial tau-functions  of the  KP hierarchy  are described as determinants of Jacobi\,-\,Trudi type. 
Below we  describe polynomial tau-functions as coefficients of certain family of  generating functions  that have  a form  similar to (\ref{QKP}).

Consider a  collection  of  formal  Laurent series  $A_1(u), \dots, A_l(u)\in \bC((u))$. Define the $\Lambda$-valued formal Laurent series 
  $T_i(u)= A_i(u)H(u)\in \Lambda((u))$, $i=1,\dots,l$, and  let $T_{i;\,k}\in \Lambda$ be the  coefficients of the expansion 
\[ T_i(u)=\sum_{k \in\bZ} T_{i;\,k} u^k\in \Lambda((u)),\quad i=1,\dots, l.
\]
Define also the formal Laurent series in the variables $u_1,\dots, u_l$
\begin{align*}
T(u_1, \dots, u_l)
=\prod_{1\le i<j\le l}\left(u_j-{u_i}\right)\prod_{i=1}^{l}A_i(u_i)  H(u_i)\, \in\Lambda[[u_1,\dots, u_l]] [u_1^{-1},\dots, u_l^{-1}].
\end{align*}
For any vector
 $\alpha=(\alpha_1,\dots, \alpha_l)\in \bZ^l$ let $T_\alpha$ be the coefficient of the corresponding monomial in the expansion 
\begin{align}\label{AQKP}
T(u_1, \dots, u_l)=\sum_{\alpha\in \bZ^l} T_\alpha u_1^{\alpha_1}\dots u^{\alpha_l}.
\end{align}

 \begin{theorem}  \label{theoremKP}
 \begin{enumerate}[label=\alph*)]
\item
\begin{align}\label{TKP}
 &T(u_1, \dots, u_l)= \det  [u_i^{j-1}T_{i}(u_i)]_{i,j=1,\dots, l}.
\end{align}
\item
For any  
$(\alpha_1,\dots, \alpha_l)\in \bZ^l$ the coefficient $T_\alpha$ of the monomial
$u_1^{\alpha_1}\dots u_l^{\alpha_l}$
is given by
\begin{align}\label{TcKP}
T_\alpha= \det[T_{i;\,\alpha_i+1-j}]_{i,j=1,\dots,l}.
\end{align}
\item
 For any  $(\alpha_1,\dots, \alpha_l)\in \bZ^l$ the  element $T_\alpha\in \Lambda= \B^{(0)}$ 
is  a tau-function of the KP hierarchy.

\item If $A_1(u), \dots, A_l(u)\in \bC[u,u^{-1}]$ are non-zero Laurent polynomials, then  the  element $T_\alpha$  is a polynomial tau-function of the KP hierarchy.

\item Let $\tau\in \B^{(0)}$ be a   polynomial  tau-function of the KP hiearchy. Then there exists  a  collection  of  Laurent  polynomials $A_1(u), \dots, A_l(u)\in \bC[u, u^{-1}]$   such that
$\tau$  is the zero-mode  of the Laurent series expansion of (\ref{AQKP}).
\end{enumerate}
\end{theorem}

\begin{proof}
Proof of
a) and b)  repeats word-to-word calculations  of \cite{JR-genA} Section 2.2: 
\begin{align*}
T(u_1, \dots, u_l)&=\prod_{i<j}  \left(u_j-{u_i}\right) \prod_{i=1}^{l} T_i(u_i)
= \det[u_i^{j-1}] \prod_{i=1}^{l} T_i(u_i)
= \det [u_i^{j-1}T_i(u_i) ]\\
&=\sum_{\sigma\in S_l} sgn(\sigma) \sum_{a_1\dots a_l} T_{1,a_1}u_1^{a_1+\sigma(1)-1}\cdots T_{l,a_l}u_l^{a_l+\sigma(l)-1}\notag\\
&=
\sum_{\alpha_1\dots \alpha_l} \sum_{\sigma\in S_l}sgn(\sigma)T_{1,\alpha_1+1-\sigma(1)}\cdots T_{l,\alpha_l+1-\sigma(l)}u_1^{\alpha_1}\cdots u_l^{\alpha_l}\notag
\\
&=\sum_{\alpha_1\dots \alpha_l}\det [T_{i, \alpha_i+1-j}]u_1^{\alpha_1}\cdots u_l^{\alpha_l}.\notag
\end{align*}
c)
 From (\ref{fermQ}) we have
\begin{align*}
A_l(u_l) \dots A_1(u_1)\psi^+(u_l) \dots \psi^+(u_1) (z^r\cdot 1)
= z^{r+l}u_l^r\dots u_1^r\,T(u_1, \dots, u_l). 
\end{align*}
Let 
$A_j(u)=\sum_{k\ge M_j}A_{j,k-1/2}u^k$
 (here $A_{j,k-1/2}\in \bC$, $M_j\in \bZ$, $k\in \bZ$, $j=1,\dots, l$).
Then the coefficient of $u_1^{\alpha_1}\dots  u_l^{\alpha_l}$  in $T(u_1, \dots, u_l)$   has   the form
$
 T_\alpha = z^{-l-r}X_l\dots X_1 (z^r\cdot 1)
$
with 
\begin{align}\label{Xi}
X_j=\sum_{i\ge M_j-\alpha_j+r-1/2}A_{j,\, \alpha_j+i-r }\psi^+_{i}.
\end{align}
In particular, take  $r=0$ to deduce that  the  coefficient of 
$u_1^{\alpha_1}\dots  u_l^{\alpha_l}$  in $T(u_1, \dots, u_l)$ 
is a tau-function of the KP hierarchy  by Corollary \ref{corKP}.
\\
\\
d)
If  $A_{j,k-1/2}=0$  for  all $i=1,\dots, l$  and  for all but finitely many $k\in \bZ$,  the expression $
X_1\dots X_l (1)
$
in c) is a finite linear combination of  elements of the form $\psi^+_{k_1}\dots \psi^+_{k_l} (1)$, hence it is a  polynomial tau-function. 
\\
\quad 
\\
e) The group $GL_\infty$ of automorphisms of the vector space  $\bC^{\infty}=\oplus_{j\in \bZ}\bC e_j$, which fix all  but finitely many $e_j$'s, acts naturally on
the semi-infinite wedge space  $\bigwedge^{\frac{1}{2}}\bC^\infty$, for  which   all semi-infinite monomials  of the form $e_{i_1}\wedge e_{i_2}\wedge e_{i_3}\wedge \dots$, 
where 
$i_1>i_2>i_3>\dots $ and  $i_{l+1}=i_l-1$ for $l>>0$, form a basis. 
Let $|m\rangle =e_m \wedge e_{m-1} \wedge ....$.
Through the boson-fermion correspondence the space $\B$ is identified with  $\bigwedge^{\frac{1}{2}}\bC^\infty$, so that $z^m$ is identified with $|m\rangle$,  see \cite{Kac-bomb} Lecture 5 for more details.  Furthermore,  the orbit of $GL_\infty\cdot 1$ in $\B$ coincides with the set of polynomial tau-functions, see \cite{Kac-bomb} Proposition 7.2. 

Let $W$ be the group  of permutations of basis vectors $e_i$'s  of the vector space $\bC^\infty$, which fix all but finitely many of them.  Let $W_0\subset W$ be  the subgroup   of permutations that fix  vectors $e_i$ with indices  $i>0$. By the Bruhat decomposition  of  $GL_\infty$ (cf.  \cite{KL-sKP}), any   element of  the orbit  of $\bC|0\rangle$ 
 has the    form
$
Bw(\lambda)\cdot|0\rangle,
$
where 
 $B= (b_{ij})_{i,j\in \bZ}\in GL_\infty$ is an upper-triangular matrix with $1$'s  on the diagonal ($b_{ij}=0$ for $i>j$, and $b_{ii}=1$),  and 
  $w(\lambda)\in W/W_0$, where  $\lambda=(\lambda_1,\dots, \lambda_l)$ is a partition, is defined by:  
 \[
 w(\lambda)\cdot|0\rangle= e_{\lambda_1}\wedge\dots \wedge e_{\lambda_l-l+1}\wedge e_{-l}\wedge e_{-l-1}\wedge \dots.
\]
Hence
 \begin{align*}
Bw(\lambda)\cdot|0\rangle&= 
Be_{\lambda_1}\wedge \dots \wedge Be_{\lambda_l-l+1} \wedge e_{-l}\wedge e_{-l-1}\wedge \dots
\\&=
f_1\wedge f_2\wedge \dots \wedge f_l \wedge e_{-l}\wedge e_{-l-1}\wedge\dots,  
\end{align*}
where 
\begin{align*}
f_j
&=\sum_{s=-l +1} ^{\lambda_j-j+1} b_{s,{\lambda_j-j+1}}e_{s}.
\end{align*}
Note that $f_j$  is  just the `truncated  at the level $-l+1$'  column $(b_{i,\lambda_j-j+1})_{i\in\bZ}$ of the matrix  $B$,   due to the operation of wedge product with
 $ |-l\rangle= e_{-l}\wedge e_{-l+1}\wedge \dots$.
 Using that $\psi^+_{-s+1/2} \left(e_{i_1}\wedge e_{i_2}\dots\right)= e_s\wedge e_{i_1}\wedge e_{i_2}\dots$, write 
  \begin{align*}
Bw(\lambda)\cdot|0\rangle&
=
Y_1\dots Y_{l}|-l\rangle, 
 \end{align*}
where
   \begin{align}\label{Xp}
  Y_{j}&=\sum_{s=-l +1} ^{\lambda_j-j+1} b_{s,{\lambda_j-j+1}}\psi^+_{-s+1/2},\quad  j=1,\dots, l.
\end{align}
For a fixed vector $(\alpha_1,\dots, \alpha_l)\in \bZ^l$
set
\[
A_{j}(u)=\sum_{t=\alpha_j-\lambda_j+j-l}^{\alpha_j} A_{j,t-1/2}u^t,\quad  j=1,\dots, l,
\]
with  non-zero terms of the sum defined by entries of the $\lambda_j-j+1$-th column of the matrix $B$:
\[
A_{j,t-1/2}=b_{-l+1+\alpha_j-t,\lambda_j-j+1}.
\]
Then 
\begin{align*}
\sum_{i}A_{j,\, \alpha_j+i +l }\psi^+_{i}=\sum_{i}b_{1/2-i,\, \lambda_j-j+1 }\psi^+_{i}=\sum_{s}  b_{s,{\lambda_j-j+1}}\psi^+_{-s+1/2},
\end{align*}
and we see that $X_j$    in  (\ref{Xi}) coincides  with $Y_j$  in  (\ref{Xp}).
Hence, with such choice of $A_1(u),\dots,  A_l(u)$, the coefficient $T_\alpha$ of $u_1^{\alpha_1} \dots u_l^{\alpha_l}$  in  $T(u_1,\dots, u_l)$  coincides with the given polynomial  tau-function  $Y_l\dots Y_{1}|-l\rangle= \pm B w(\lambda)\cdot|0\rangle$.
Note that we proved that for any polynomial tau-function $\tau$ and any given vector  $(\alpha_1,\dots, \alpha_l)\in \bZ^l$ of the appropriate length $l$   there exists $T(u_1,\dots, u_l)$ such that  $\tau$ is the  coefficient of  $u_1^{\alpha_1} \dots u_l^{\alpha_l}$ in $T(u_1,\dots, u_l)$.  In particular, we can consider $\alpha_1=\dots =\alpha_l=0$ to  represent  the given polynomial tau-function as the zero-mode of certain $T(u_1,\dots, u_l)$.
\end{proof}
\subsection{Remarks and corollaries.}\label{rcKP}

\begin{enumerate}[label=\alph*)]
  	\item
 We proved that  polynomial tau-functions  are  zero-modes  of appropriate generating functions $T(u_1,\dots, u_l)$,   but   changing back
$A_j(u)\mapsto u^{\alpha_j} A_j(u)$  allows  one  to get any  polynomial tau-function as a coefficient of a  given monomial $u_1^{\alpha_1}\dots u_l^{\alpha_l}$.
\item 
	 Consider non-zero Laurent series  $A_1(u), \dots, A_l(u)$ in the definition of $T(u_1,\dots, u_l)$.
One can represent each of  them in the form 
	\[
	A_j(u)= u^{M_j} b_{j}\sum_{i=0}^{\infty} a_{j, i} u^i
	\]
for suitable choice of $M_j\in \bZ$, $b_j, a_{j,i}\in \bC$, and  $a_{j,0}=1$, $b_j\ne 0$.
Note that there exists a collection of constants $\{c_{j,s}\}\subset \bC$ such that 
$$
\sum_{i=0}^{\infty} a_{j, i} u^i = exp\left(\sum_{s=1}^{\infty} c_{j,s} u^s\right),
$$
in other words,  $ a_{j, i}= S_i(c_{j,1}, c_{j,2},\dots)$. Then, letting $t_k=kp_k$, we obtain by (\ref{HEP})
\begin{align*}
T_j(u)&= A_j(u)H(u)
=u^{M_j} b_{j}exp\left(\sum_{k=1}^{\infty} c_{j,k} u^k\right)exp\left(\sum_{k=1}^{\infty}t_ku^k\right)\\
&=u^{M_j} b_{j}exp\left(\sum_{k=1}^{\infty}\left(t_k+ c_{j,k}\right) u^k\right)= u^{M_j} b_{j}\sum_{k=0}^{\infty}
S_k(t_1+c_{j,1},\,t_2+c_{j,2},\dots) u^k.
\end{align*}
Hence
$
T_{j;k}=b_jS_{k-M_j}(t_1+c_{j,1},\,t_2+c_{j,2},\dots),
$
and by   Theorem \ref{theoremKP}, polynomial KP tau-functions  have the form
\[
T_\alpha=\left(\prod_k b_k\right) \cdot \det[ S_{\alpha_i-M_i+1-j}(t_1+c_{i,1},\,t_2+c_{i,2},\dots)]_{i,j=1,\dots,l}
\]
for any choice of  constants $\{c_{j,k}\}$. This recovers  the description of all polynomial tau-functions of the KP hierarchy   of 
\cite{KL-KP}  and, in particular,  Sato's theorem \cite{Sato}
that all Schur polynomials  are   tau-functions of the KP hierarchy. Obviously, it suffices to take here  for $A_j(u)$ Laurent polynomials. 

\end{enumerate}

\section{Polynomial  tau-functions of the  BKP  hierarchy}\label{Sec3}

In this section we describe polynomial tau-functions  of the  BKP  hierarchy as coefficients of certain  generating functions.

\subsection{Neutral fermions action on the boson Fock space}
Through  all of Section \ref{Sec3} denote
$ Q(u)= E(u) H(u)=S(u)^2=\sum_{k\in \bZ}q_k u^k $
 as in (\ref{shurq}) and (\ref{QS}).
Consider the  boson Fock space  generated by odd power sums:
  \[\B_{odd}= \bC[  p_1, p_3, p_5,\dots].
   \]
    Recall   (\cite{Md}, III.8 (8.5))  that   $q_k\in \B_{odd}$,  and  that $\B_{odd}= \bC[q_1,q_3,\dots ]$.
  From (\ref{DHEP}) it is clear that  $\B_{odd}$  is invariant with respect  to action of  $e^\perp_k$,  and $ h^\perp_k$, and 
one can   prove   \cite{R-Q}
that restrictions to  ${\B_{odd}}$ of the operators
$
 E^{\perp}(u)$,  $H^{\perp}(u)$, $S^{\perp}(u)$
coincide.
We can  define the following  formal distribution of operators acting on $\B_{odd}$:
 \begin{align}\label{defphi}
&\varphi(u)=  E(u)H(u) E^{\perp}(-u)=Q(u) S^{\perp}(-u).
\end{align}
Let  $\{\varphi_i\}_{i\in \bZ}$ be
coefficients  of the expansion $
\varphi(u)=\sum_{j\in \bZ}\varphi_j u^{-j}$.
Similarly to Proposition \ref{fermR} one proves  the  following statements \cite{R-Q}.
\begin{proposition}\label{nferm} 

  \begin{enumerate}[label=\alph*)]
\item
Formula (\ref{defphi})  defines a quantum field $\varphi(u)$ of  operators acting on  the space  $\B_{odd}$.
\item
One has relations
\[
 \varphi(u)\varphi(v) + \varphi(v)\varphi(u) =2v\delta(v,-u),
 \]
where 
$\delta(u,v)$ is  the formal delta function. 
Hence (\ref{defphi}) is the action of  the Clifford algebra of neutral fermions on the space $\B_{odd}$:
 \begin{align}\label{neut1}
\varphi_m \varphi_n+\varphi_n \varphi_m= 2(-1)^m \delta_{m+n,0}\quad \text {for}\quad m, n\in \bZ.
\end{align}

\end{enumerate}
\end{proposition}
\begin{remark}\label{41}
Using (\ref{HEP}), (\ref{DHEP})  one gets the bosonic form of   
 $\varphi(u)$:

\begin{align*}
\varphi(u)=  Q(u) S(-u)^{\perp}= exp\left(\sum_{n\in \bN_{odd}}\frac{2p_{n}}{n}{u^{n}}\right)exp\left(-\sum_{n\in \bN_{odd}} \frac{\partial}{\partial p_n}\frac{1}{u^{n}}\right).
\end{align*}
It follows that $\varphi_n(1)=0$ for $n>0$ and $\varphi_0(1)=1$.
\end{remark}

\subsection{The bilinear  BKP identity.}    The {\it bilinear BKP identity} \cite{DJKM3}, \cite{DJKM2}, \cite{KL-BKP} is the  equation  of the form 
\begin{align}\label{BKPid}
\Omega ( \tau\otimes \tau)= \tau\otimes\tau
\end{align}
on elements  $\tau=\tau(p_1,p_3, p_5\dots)$ from the completion of $\B_{odd}$,
where 
\begin{align*}
\Omega=\sum_{n\in \bZ} \varphi_n\otimes (-1)^n \varphi_{-n}.
\end{align*}
Non-zero solutions of (\ref{BKPid}) are called the  {\it tau-functions of the BKP  hierarchy}. We will say that a   solution of (\ref{BKPid})  is  a polynomial 
tau-function if  it is a polynomial function  in the  variables  $(p_1,p_3, p_5\dots)$ (hence, it is an element of  $\B_{odd}$ rather than its completion).
By Remark \ref{41}, $\tau=1$ is a tau-function of the BKP hierarchy.
 Similarly to  Lemma \ref{OmXKPlemma}  and statements of \cite{R-Q} we prove the following lemma. 

\begin{lemma} \label{OmXBKPlemma1}
Let  $X=\sum_{n\ge M}A_n\varphi_n$, where $A_n\in \bC$ and $M\in \bZ$.  
Then
\begin{align}\label{x2}
X^2= \sum_{M\le k\le -M} (-1)^kA_k A_{-k}\cdot Id,
\end{align}
in particular,   $X^2= A_0^2\cdot Id$  if $M=0$, and $X^2= 0$ if $M>0$.
\end{lemma}

\begin{proof}
Note that due to Proposition \ref{nferm}  part  a)  $X$ and $X^2$ are well-defined operators on the space $\B_{odd}$.
Split   $X=X^{-}+ A_0\varphi_0 +X^+$, where
$X^{-}=\sum_{ n<0}A_n\varphi_n$ , $X^+=\sum_{n>0}A_n\varphi_n$.
Due to commutation relations (\ref{neut1}),
\begin{align*}
(X^{\pm})^2 =0, \quad \varphi_0 (X^++X^-)+ (X^++X^-)\varphi_0=0,\quad \varphi_0^2= Id, 
\end{align*}
and 
\begin{align*}
\quad X^+X^-+ X^-X^+ =\sum_{M\le n<0}2 (-1)^nA_n A_{-n}\cdot Id=2\sum_{\substack{M\le n\le -M\\ n\ne 0}} (-1)^nA_n A_{-n}\cdot Id.
\end{align*}
 if $M<0$, and $X^+X^-+ X^-X^+=0$ if $M\ge 0$.
Applying these identities in the expansion of $X^2$ we get
\begin{align*}
X^2= (X^{-}+ A_0\varphi_0 +X^+)^2
=\sum_{k=M}^{-M} (-1)^kA_k A_{-k}\cdot Id.
\end{align*} 
\end{proof}

\begin{lemma} \label{OmXBKPlemma2}
	\begin{align}\label{omx}
		\Omega (X\otimes X)= (X\otimes X)\Omega. 
	\end{align}
\end{lemma}
\begin{proof}
For any $n\in\bZ$,
\[
\varphi_n X+X\varphi_n= 2(-1)^n\sum_{k\ge M} A_k\delta_{k+n,0}=
\begin{cases}
2(-1)^n A_{-n},& \text{if}\quad n\le -M,\\
0&\text{otherwise}.
\end{cases}
\]
Then 
\begin{align*}
\Omega (X\otimes X)&= (X\otimes X) \Omega -2 \sum_{n\in\bZ}\sum_{l\ge M}  X\varphi_n\otimes A_l\delta_{l-n,0}
-2 \sum_{n\in\bZ}\sum_{k\ge M} A_k\delta_{k+n,0}\otimes X\varphi_{-n} \\
 +&4 \sum_{n\in\bZ}\,\sum_{k,l\ge M}(-1)^n A_k\delta_{k+n,0}\otimes A_l\delta_{l-n,0}\\
 &= (X\otimes X) \Omega -2(X^2\otimes 1+ 1\otimes X^2) 
 + \left(4\sum_{M\le k\le -M} (-1)^k A_k A_{-k} \right)1\otimes 1.
\end{align*}
Using (\ref{x2}) we complete the proof of (\ref{omx}).

\end{proof}
 \begin{corollary}\label{corBKP}
Let $\tau\in \B_{odd}$  be a tau-function of the BKP hierarchy,  and let  $X=\sum_{n\ge M}A_n\varphi_n$, where $A_n\in \bC$ and $M\in \bZ$.  
Then   $\tau^\prime =X\tau$ is also a tau-function of the BKP hierarchy. 
\end{corollary}

\subsection{Generating functions for polynomial tau-functions  of the BKP hierarchy}

 Let 
 \begin{align*}
 f(u,v)=i_{u,v}\left(\frac{u-v}{u+v}\right)=\left(1-\frac{v}{u}\right)\sum_{k\in \bZ_+}(-1)^k\frac{v^k}{u^k} =1+2\sum_{k\in \bN}(-1)^k\frac{v^k}{u^k}\in \bC[[v/u]].
 \end{align*}
 Note that 
 \begin{align}\label{fuv}
 f(u,v)+f(v,u)=(u-v)\delta(u,-v) =2\sum_{k\in \bZ}\frac{u^k}{(-v)^k}.
 \end{align} 
Define a formal distribution $Q(u_1,\dots, u_l)$ with coefficients in  $\B_{odd}$ 
as a result of application of a product of quantum fields $\varphi(u_i)$ to the vacuum vector $1\in \B_{odd}$.
\begin{align}\label{defQBKP}
\varphi(u_l)\dots \varphi(u_1) (1)= Q(u_1,\dots, u_l),
\end{align}
By  (\ref{defphi}) and  Lemma \ref{propHE} we get
\begin{align*}
Q(u_1,\dots, u_l)=
\prod_{1\le i<j\le l} f(u_j,u_i)\prod_{i=1}^{l} Q(u_i)
\end{align*}
 where   $Q(u)$  is the  generating series (\ref{shurq}). 
Define the elements  $Q_\alpha\in \B_{odd}$ as the coefficients of the formal distribution 
$
Q(u_1,\dots, u_l)=\sum_{\alpha\in \bZ^l}Q_\alpha u_1^{\alpha_1}\dots u_l^{\alpha_l}.
$
Note that 
(compare with  Theorem 4.28 in \cite{Jing1})\begin{align*}
Q_\alpha= \varphi_{-\alpha_1}\dots \varphi_{-\alpha_l} (1)
\end{align*}

Let $A_1(u),\dots, A_l(u)\in \bC[u, u^{-1}]$ be a collection of   Laurent polynomials. 
Define a $\B_{odd}$\,-\,valued formal distribution 
\begin{align}\label{AQBKP}
T(u_1, \dots, u_l)= \prod_{i=1}^{l}A_i(u_i) \,Q(u_1, \dots, u_l).
\end{align}
For any  $\alpha=(\alpha_1,\dots, \alpha_l)\in \bZ^l$ let $T_\alpha$ be the coefficient in the expansion 
\begin{align}\label{AQBKP1}
T(u_1,\dots, u_l)=\sum_{\alpha\in \bZ^l}T_\alpha u_1^{\alpha_1}\dots u_l^{\alpha_l}.
\end{align}

\begin{remark}
\label{r1}
Note that for any choice of   Laurent polynomials  $A_1(u),\dots, A_{2l-1}(u)$ and with $A_{2l}(u)=1$ in (\ref{AQBKP})  we have 
\[T(u_1,u_2,\dots, u_{2l-1})=-T(u_1,u_2,\dots, u_{2l-1}, 0).\]
 Hence, it is sufficient  to consider only the case of even number of  variables 
$(u_1,\dots, u_{2l})$.
\end{remark}

Let $(u_1, u_2,\dots, u_l)$ be a collection of variables.  
Introduce the matrix $\tilde F= [\tilde f_{i,j}]_{i,j=1,\dots, 2l}$ with coefficients 
\[\tilde f_{i,j}= 
\begin{cases}
\quad f(u_i,u_j),&  \text{if} \quad i<j,\\
\quad 0,&  \text{if} \quad i=j,\\
-f(u_j,u_i),&  \text{if} \quad i>j.\\
\end{cases}
\]
For $  i,j\in\{1,\dots, 2l\}$ consider  formal distributions  
\[
T^{(i)}(u_i)= A_i(u_i) Q(u_i)
\quad \text{and}\quad 
\tilde T^{(i,j)}(u_i, u_j)=\tilde f_{i,j} T^{(i)}(u_i) T^{(j)}(u_j).
\]
Note that $ \tilde T^{(i,j)}(u_i, u_j)=- \tilde T^{(j,i)}(u_j, u_i)$. 

For any $(a,b)\in \bZ^2$ and $i<j$
 let  $\tilde T^{(i,j)}_{a,b}$  be  the coefficients of the expansion
\[
\tilde T^{(i,j)}(u_i, u_j)=\sum_{a,b\in\bZ}\tilde T^{(i,j)}_{a,b}u_i^au_j^b. 
\]
 For $i>j$  set  $\tilde T^{(i,j)}_{a,b}= -\tilde T^{(j,i)}_{a,b}$, and  set  $\tilde T^{(i,i)}_{a,b}=0$.

 \begin{theorem}  \label{ThmBKP}
\begin{enumerate}[label=\alph*)]
\item
\begin{align*}
T(u_1,u_2,\dots, u_{2l})=\mathrm{Pf}\left[\tilde T^{(i,j)} (u_i, u_j)\right]_{i,j=1,\dots, 2l}.
\end{align*}
\item 
For any 
$(\alpha_1,\dots, \alpha_{2l})\in \bZ^{2l}$(see Remark \ref{r1} for the odd case) the coefficient  $T_\alpha$ of the monomial
$u_1^{\alpha_1}\dots u_{2l}^{\alpha_{2l}}$ in (\ref{AQBKP1})   equals 
\[
	T_\alpha= \mathrm{Pf} [\tilde T^{(i,j)}_{\alpha_i,\alpha_j}]_{i,j=1,\dots, 2l}.
\]

\item
For  any  $(\alpha_1,\dots, \alpha_{l}) \in \bZ^l$   the  coefficient  $T_\alpha$ of  $u_1^{\alpha_1}\dots u_l^{\alpha_l}$ in (\ref{AQBKP1})
is  a  polynomial tau-function of the BKP hierarchy. 

\item 
Let $\tau$ be a  polynomial  tau-function of the BKP hierarchy.  Then there exists  a collection  of Laurent polynomials  $A_1(u), \dots, A_l(u) $ 
 such that $\tau$ is the zero-mode of the series expansion of (\ref{AQBKP}).
\end{enumerate}
\end{theorem}
\begin{proof}  
 The proof of of a) and b) follows the steps of  the  proof of similar statement in \cite{JR-genA}, Section 2.3. 
Recall (\cite{Md}, III.8) the equality of rational functions
\begin{align}\label{Pf}
\mathrm{Pf} \left [\frac{u_i-u_j}{u_i+u_j} \right] _{i,j=1,\dots, 2l}=\prod_{1\le i<j\le 2l} \frac{u_i-u_j}{u_i+u_j}.
\end{align}

Let     $ i ( g(u_1,\dots, u_l))$  be the series expansion of the  rational function $g(u_1,\dots, u_l)$  in the variables $u_1, \dots, u_l$ in the region
 $|u_1|> \dots > |u_l|$. Then 
 $i\left(\frac{u_i-u_j}{u_i+u_j}\right) = \tilde f(u_i, u_j)$, 
and 
 \[
   i\left( \mathrm{Pf}\,\left[\frac{u_i-u_j}{u_i+u_j}\right]\right)=  \sum_{\sigma\in S'_{2l}}  sgn(\sigma) \tilde f_{\sigma(1) \sigma(2)} \dots  \tilde f_{\sigma(2l-1) \sigma(2l)}= \mathrm{Pf}\,\tilde F.
 \]
 Expanding the rational  function in the right-hand side of (\ref{Pf}) we get 
 \[
 \mathrm{Pf}\, \tilde F= \prod_{1\le i<j\le 2l}f(u_i, u_j).
 \]
 Then 
\begin{align*}
T(u_1,u_2,\dots, u_{2l})&=\prod_{1\le i<j\le 2l}f(u_i, u_j)\prod_{i=1}^{2l} T^{(i)}(u_i)=
 \mathrm{Pf}\, \tilde F\, \prod_{i=1}^{2l} T^{(i)}(u_i)\\
 &=  \sum_{\sigma\in S'_{2l}} sgn(\sigma) \tilde f_{\sigma(1)\sigma(2)} \cdots \tilde f_{\sigma(2l-1)\sigma(2l)} \prod_{i=1}^{2l} T^{(i)}(u_i)\\
 &=  \sum_{\sigma\in S'_{2l}}sgn(\sigma)  \tilde f_{\sigma(1)\sigma(2)} T^{\sigma(1) }(u_{\sigma(1) }) T^{\sigma(2) }(u_{\sigma(2) }) \cdots \tilde  f_{\sigma(2l-1)\sigma(2l)} T^{\sigma(2l-1) }(u_{\sigma(2l-1) }) T^{\sigma(2l) }(u_{\sigma(2l) })\\
 &=\mathrm{Pf} [\tilde f_{ij} T^{(i)}(u_i) T^{(j)}(u_j)]_{i,j=1,\dots, 2l}=\mathrm{Pf} [\tilde T^{(ij)}(u_i, u_j)]_{i,j=1,\dots, 2l}.
\end{align*}

b) 
Observe that 
\begin{align*}
\mathrm{Pf} &[\tilde T^{(ij)}(u_i, u_j)  ]_{i,j=1,\dots, 2l}=
\mathrm{Pf}\left[ \sum_{\alpha_i,\alpha_j} {\tilde T^{(i,j)}_{\alpha_i,\alpha_j}}  u_i^{\alpha_i}  u_j^{\alpha_j} \right]_{i,j=1,\dots, 2l}\\
&=  \sum_{\sigma\in S'_{2l}}sgn(\sigma)  \sum_{\alpha_i,\alpha_j}\tilde T^{(\sigma(1),\sigma(2))}_{\alpha_{\sigma(1)},\alpha_{\sigma(2)}}  u_{\sigma(1)}^{\alpha_{\sigma(1)} }
 u_{\sigma(2)}^{\alpha_{\sigma(2) }}\cdots  \tilde T^{(\sigma(2l-1),\sigma(2l))}_{\alpha_{\sigma(2l-1)},\alpha_{\sigma(2l)}} u_{\sigma(2l-1)}^{\alpha_{\sigma(2l-1)} }
 u_{\sigma(2l)}^{\alpha_{\sigma(2l)} }
\\
&=  \sum_{\sigma\in S'_{2l}} sgn(\sigma) \sum_{\alpha_i,\alpha_j}\tilde T^{(\sigma(1),\sigma(2))}_{\alpha_{\sigma(1)},\alpha_{\sigma(2)}}  \cdots  \tilde T^{(\sigma(2l-1),\sigma(2l))}_{\alpha_{\sigma(2l-1)},\alpha_{\sigma(2l)}} u_{1}^{\alpha_{1} }\dots
 u_{2l}^{\alpha_{2l} }
=\sum_\alpha \mathrm{Pf} [\tilde T^{(i,j)}_{\alpha_i,\alpha_j}]u_1^{\alpha_1}\cdots u_{2l}^{\alpha_{2l}}.
\end{align*}
Therefore,  the  coefficient   of $u_1^{\alpha_1}\cdots u_{2l}^{\alpha_{2l}}$ in  (\ref{AQBKP})  equals
$
 \mathrm{Pf} [\tilde T^{(i,j)}_{\alpha_i, \alpha_j}]_{i,j=1,\dots 2l}.
$

c)
Let $A_{i}(u)=\sum_{N_i\ge j\ge M_i}A_{i,j}u^j\in \bC[u, u^{-1}]$,  $i=1,\dots, l$, $N_i,M_i\in \bZ$.  Then
 from (\ref{defQBKP}), 
  \begin{align*}
T(u_1, \dots, u_l)
 &=
 A_1(u_1)\varphi(u_1)\dots  A_l(u_l)\varphi(u_l) (1)\\
 &=\sum_{N_1\ge j_1\ge M_1}\sum_{ k_1\in \bZ}A_{1,j_1}\varphi_{k_1} u_1^{j_1-k_1}\dots \sum_{N_l\ge j_l\ge M_l}\sum_{ k_l\in \bZ}A_{l,j_l}\varphi_{k_l} u_l^{j_l-k_l}(1)\\
  &=\sum_{\alpha\in \bZ,^l}\sum_{ N_1-\alpha_1 \ge k_1\ge M_1-\alpha_1}A_{1, \alpha_1+k_1}\varphi_{k_1} u_1^{\alpha_1}\dots \sum_{N_l-\alpha_l\ge k_l\ge M_l-\alpha_l}A_{l, \alpha_l+k_l}\varphi_{k_l} u_l^{\alpha_l}(1).
  \end{align*}
 Then the coefficient $T_\alpha$ of 
$u_1^{\alpha_1}\dots  u_l^{\alpha_l}$ is  equal to  
$
X_1\dots X_l \,(1)$,
where
\begin{align}\label{Xbkp}
X_j=\sum_{N_j-\alpha_j\ge k\ge M_j-\alpha_j}A_{j,\alpha_j+k}\varphi_{k}.
\end{align}
By Corollary \ref{corBKP} and Remark \ref{41},  the coefficient  $T_\alpha$  is a  tau-function of the BKP hierarchy.
Since it is a finite linear combination of  elements of the form $\varphi_{k_1}\dots \varphi_{k_l}(1)$, it is  a polynomial tau-function.

d)
Following \cite{KL-BKP} Proposition 3,
any  polynomial tau-function has the form 
\begin{align*}
\tau= X_1\dots X_l (1),
\end{align*}
 where $\lambda=(\lambda_1>\lambda_2>\dots>\lambda_l\ge 0)$  is a strict partition, and 
\begin{align}\label{taubkp}
X_j=\sum_{N_j\ge i\ge-\lambda_j} b_{i,j} \varphi_i
\end{align}
with $b_{i,j}\in\bC$, and $b_{-\lambda_j, j}\ne0$, $N_j\in \bZ$ for  $j=1,\dots, l$.
Fix $(\alpha_1,\dots, \alpha_l)\in \bZ^l$. In the definition of $T(u_1,\dots, u_l)$ set the Laurent polynomial 
 \begin{align*}
 A_{j}(u)&=\sum_{N_j+\alpha_j\ge k\ge \alpha_j-\lambda_j}b_{k-\alpha_j,j} u^{k}, \quad j=1,\dots, l.
\end{align*}
With such  a choice of $A_1(u), \dots, A_l(u)$  the operator  (\ref{Xbkp}) becomes
exactly of the form  (\ref{taubkp}), and the coefficient $T_\alpha$  coincides with the given  polynomial  tau-function $\tau$. 
In particular, take $\alpha_1=\dots =\alpha_l=0$  to represent $\tau$ as 
the constant  coefficient of some $T(u_1,\dots, u_l)$.
\end{proof}

\subsection{Remarks and corollaries} Theorem \ref{ThmBKP} implies several corollaries.
 \begin{enumerate}[label=\alph*)]
 \item
 We proved that  polynomial tau-functions  are  zero-modes  of appropriate generating functions $T(u_1,\dots, u_l)$,   but   changing back
$A_j(u)\mapsto u^{\alpha_j} A_j(u)$  allows  one  to get any  polynomial   tau-function as a coefficient of a monomial $u_1^{\alpha_1}\dots u_l^{\alpha_l}$ of any degree 
$(\alpha_1,\dots, \alpha_l)$.

 \item 

  Consider non-zero Laurent polynomials   $A_1(u), \dots, A_l(u)$ in the definition of $T(u_1,\dots, u_l)$. Due to remark above we can assume without loss of generality 
  they are polynomials
	\[
	A_j(u)= b_{j}\sum_{i=0}^{N_j} a_{j, i} u^i \in \bC[u]
	\]
where  $ N_j \in \bZ$, $b_j, a_{j,i}\in \bC$, and  $a_{j,0}=1$.
Note that there exists a collection of constants $\{c_{j,s}\}\subset \bC$ such that 
$$
\sum_{i=0}^{N_j} a_{j, i} u^i = exp\left(\sum_{s=1}^{\infty} c_{j,s} u^s\right),
$$
in other words,  $ a_{j, i}= S_i(c_{j,1}, c_{j,2},\dots)$. Letting $(\tilde t_1,\tilde t_2,\tilde t_3\dots) =(p_1/2, 0,3p_3/2,0,\dots ) $, we obtain by (\ref{QS})
\begin{align*}
T^{(j)}(u_j)&= A_j(u_j) Q(u_j)
= b_{j}exp\left(\sum_{k\in \bN} c_{j,k} u_j^k\right)exp\left(\sum_{k\in \bN}\tilde t_ku_j^k\right)\\
=& b_{j}\sum_{k\in \bN}
S_k(\tilde t_1+c_{j,1},\,\tilde t_2+c_{j,2},\dots) u_j^{k},
\end{align*}
Then for $i<j$
\[
\tilde T^{(i,j)}=b_ib_j\left(1+2\sum_{k\in \bN}\left(\frac{u_j}{-u_i}\right)^k\right) \sum_{a,b\in \bN} S_{a}(\tilde t+c_{i})S_{b}(\tilde t+c_{j})u_i^{a}u_j^{b}.
\]
and
\[
\tilde T^{(i,j)}_{a,b}=2b_ib_j \chi_{a,b}(\tilde t+ c_i, \tilde t+ c_j), \]
where 
\[\chi_{a,b}(\tilde t+ c_i, \tilde t +c_j)= \frac{1}{2}S_{a}(\tilde t+c_{i})S_{b}(\tilde t+c_{j})  +\sum_{k\in \bN}(-1)^kS_{a-k}(\tilde t+c_{i})S_{b+k}(\tilde t+c_{j})
\]
for $a<b$, $\chi_{a,b}=-\chi_{b,a}$ for $a>b$, and  $\chi_{a,a} =0$.
By   Theorem \ref{ThmBKP}, polynomial  tau-functions of the BKP hierarchy  have the form
\[
T_\alpha= 2^{2l}\cdot\left( \prod_{k}  b_k^2 \right)\cdot  \mathrm{Pf} [ \chi_{a,b}(\tilde t+ c_i, \tilde t+ c_j)]_{i,j=1,\dots, 2l}
\]
for any choice of  constant vectors $c_{j}= (c_{j,k})_{k\in\bN}$. This recovers  the description of all polynomial tau-functions of the BKP hierarchy   of 
\cite{KL-BKP}  and, in particular, the result of You  \cite{You1} that   Schur  Q-polynomials are polynomial tau-functions of the BKP hierarchy.

\end{enumerate}

\section{Polynomial $\tau$-functions of the bilinear $s$-component KP hierarchy}
\label{Sec4}
In this section we consider the $s$-component KP hierarchy  which generalizes Section \ref{Sec2} that treats the case $s=1$.  Fix a positive integer 
$s$.
Following \cite{IJS}, \cite {KL-2003} and \cite{KL-sKP},
consider $s$ copies  $\Lambda ^{(1)},\dots,  \Lambda^{(s)} $ of  the algebra of symmetric functions $\Lambda$.
 We assume that  for  $a\ne b$ the operators $H^{(a)}$, $E^{(a)}$,    $H^{\perp(a)}$, $E^{\perp (a)}$, etc.   acting on the  copy $\Lambda^{(a)}$, 
 commute with the   operators  $H^{(b)}$, $E^{(b)}$,    $H^{\perp (b)}$, $E^{\perp (b)}$, etc., acting on  $\Lambda^{(b)}$.

\subsection{Fermionic fields on the $s$-component boson Fock space}
Introduce the boson Fock space $\B^{\otimes s}$ as the tensor product of $s$  copies of the algebra $\B$ from Section \ref{3.1}:
\[\B^{\otimes s} = \bC[ z_a, z^{-1}_a;  p^{(a)}_1, p^{(a)}_2, p^{(a)}_3,\dots; 1\le a\le s],
   \]
 where  $\{p^{(a)}_k\}_{a=1,\dots, s}$ are collections of  power sums in $s$ different sets of  variables, $p^{(a)}_k \in \Lambda^{(a)}$.
 We have the charge decomposition
  \[\B^{\otimes s} =\oplus_{m\in \bZ} \, \B^{\otimes s}\,^{(m)},\] 
 where 
 \[
 \B^{\otimes s}\,^{(m)}=\sum_{\substack{(m_1, \dots, m_s)\in \bZ^s,\\ m_1+\dots +m_s=m }} z_1^{m_1}\dots z_s^{m_s}\, \bC[ p^{(a)}_1, p^{(a)}_2, p^{(a)}_3,\dots].
  \]

Define the action of  operators  $R^{\pm(a)}(u)$ on $\B^{\otimes s}$  by
\[
R^{\pm(a)}(u) (z_1^{m_1}\dots z_s^{m_s} f)=
(-1)^{m_1+\dots +m_{a-1}}
z_1^{m_1}\dots z^{m_a\pm1}_a\dots  z_s^{m_s} u^{\pm m_a}f
\] 
for $(m_1,\dots, m_s)\in \bZ^s$, $f\in \bC[p^{(a)}_1, p^{(a)}_2, p^{(a)}_3,\dots; 1\le a\le s]$, which change the charge by $\pm 1$.

It is straightforward to check the following lemma. 
\begin{lemma}\label{RS}
For $a\ne b$, 
\begin{align*}
R^{\pm(a)}(u) R^{\pm(b)}(v) +R^{\pm(b)}(v) R^{\pm(a)}(u) &=0,\\
R^{+(a)}(u) R^{-(b)}(v) +R^{-(b)}(v) R^{+(a)}(u) &=0,\\
vR^{\pm(a)}(u) R^{\pm(a)}(v) -uR^{\pm(a)}(v) R^{\pm(a)}(u) &=0,\\
uR^{+(a)}(u) R^{-(a)}(v) -vR^{-(a)}(v) R^{+(a)}(u) &=0.
\end{align*}
\end{lemma}
Let $\mathcal{D}^{\otimes s}$ be the algebra of differential operators with polynomial coefficients acting on $\B^{\otimes s}$ 
(cf. Section \ref{secGKP}).
Similarly to (\ref{deff1}), (\ref {deff2}),
define for  $a\in\{1, \dots, s\}$ a collection of $\mathcal{D}^{\otimes s}$-valued  formal distributions
\begin{align}
\psi^{+(a)}(u)&=R^{+(a)}(u)H^{(a)}(u) E^{(a)\perp}(-u),\label{mP1}\\
\psi^{-(a)}(u)&=R^{-(a)}(u) E^{(a)}(-u) H^{(a)\perp}(u).\label{mP2}
\end{align}
Define   the operators $\{\psi^{\pm(a)}_i\}_{i\in \bZ+1/2}$ as  the coefficients of expansions
 \[
\psi ^
{\pm(a)}(u)= \sum_{i\in \bZ+1/2}\psi^{\pm(a)}_i u^{-i-1/2},\quad  a=1,\dots s.
\] 
Using  Lemmas \ref{RS} and   \ref{propHE}, along the same lines as  the proof of  Proposition \ref{fermR},  we obtain the commutation relations 
for  $\psi^{\pm(a)}(u)$:
\begin{proposition} 
\label{ferm_multi}
\begin{enumerate}[label=\alph*)]
\item  Formulae (\ref{mP1}), (\ref{mP2}) define 
 quantum fields $\psi^{\pm(a)}(u)$   of operators acting on  $\B^{\otimes s}$.
\item
For  all $a, b=1,\dots, s$,
\begin{align}
\psi^{\pm(a)}(u)\psi^{\pm (b)}(v)+ \psi^{\pm(b)}(v)\psi^{\pm (a)}(u)&=0, \label{mf1}
\\
\psi^{+(a)}(u)\psi^{-(b)}(v)+ \psi^{-(b)}(v)\psi^{+(a)}(u)&=\delta_{a,b}\delta(u,v).\label{mf2}
\end{align}
Relations  (\ref{mf1}), (\ref{mf2}) are equivalent to 
\begin{align}
\psi_k^{\pm(a)}\psi_l^{\pm(b)}& +\psi_l^{\pm(b)}\psi_k^{\pm(a)}=0,  \label{mf3}\\
\psi_k^{+(a)}\psi_l^{-(b)} &+\psi_l^{-(b)}\psi_k^{+(a)}=\delta_{k,-l}\delta_{a,b},\quad k,l\in \bZ+1/2. \label{mf4}
\end{align}
\item Let $\varepsilon_{ab}=1$ if $b\le a $ and $=-1$ if $b>a$, then 
\begin{equation}\label{57}
z_a\psi^{\pm (b)}_n=\varepsilon_{ab}\psi^{\pm (b)}_{n\mp\delta_{a,b}} z_a .
\end{equation}
\end{enumerate}
\end{proposition}

\begin{remark}
From  (\ref{HEP}), (\ref{DHEP}) the bosonic form of  $\psi^{\pm(a)}(u)$ is 
\begin{align*}
\psi^{+(a)}(u)&=R^{+(a)}(u)\exp \left(\sum_{n\ge 1}\frac{p^{(a)}_n}{n}{u^n}  \right) \exp \left(-\sum_{n\ge 1}\frac{\partial} {\partial p^{(a)}_n}\frac{1} {u^{n}}\right),
\\
\psi^{-(a)}(u)&=R^{-(a)}(v) \exp \left(-\sum_{n\ge 1}\frac{p^{(a)}_n}{n} {u^n}  \right) \exp \left(\sum_{n\ge 1}\frac{\partial} {\partial p^{(a)}_n} \frac{1}{u^{n}}\right)
\end{align*}
for $a=1,\dots, s$. It follows that $\psi^{+(a)}_{k-\ell_a}(z_1^{\ell_1}\dots z_m^{\ell_m})=0$ and  $\psi^{-(a)}_{k+\ell_a}(z_1^{\ell_1}\dots z_m^{\ell_m})=0$ for $k>0$.
\end{remark}
\subsection {Bilinear  $s$-component  KP identity}  
The  {\it  bilinear $s$-component KP   identity} \cite {KL-2003}, \cite{KL-sKP}  is the  equation 
\begin{align}\label{multibKP}
\Omega ( \tau\otimes \tau)= 0
\end{align}
on a function 
$
\tau=\tau(z_1,\dots, z_s, p^{(1)}_1,\dots, p^{(s)}_1,p^{(1)}_2,\dots, p^{(s)}_2,\dots)$ from the completion of   $ \B^{\otimes s}\, ^{(m)}
$,
where 
\begin{align*}
\Omega=\sum_{a=1}^{s}\Omega^{(a)},\quad 
\Omega^{(a)}=\sum_{k\in \bZ+\frac{1}{2}} \psi_k^{+(a)}\otimes \psi^{-(a)}_{-k}.
\end{align*}
 Non-zero solutions of (\ref{multibKP}) are called {\it  tau-functions  of  charge} $m$ of  the  $s$-component KP  hierarchy. We  say that a non-zero solution  $\tau$ of (\ref{multibKP})   is a
  polynomial   tau-function  if $\tau=\sum_{\substack{a\in \bZ^s,|a|=m} } z^af_a\in  \B^{\otimes s}\, ^{(m)}$, where $f_a$ are polynomial functions in the variables  $(p^{(1)}_1,\dots, p^{(s)}_1,p^{(1)}_2,\dots, p^{(s)}_2,\dots)$ and the sum is finite.
The vacuum vector $1$, as well as $z^a$, are  obviously     tau-functions of  the  $s$-component KP  hierarchy.
Similarly to Section \ref{subKPbil}, we use a family  of 
 commuting with $ \Omega$ operators  to construct  other   tau-functions.  
\begin{remark} \label{rem_22}
Note that, by (\ref{57}),  $z_i\otimes z_i$ commutes with $\Omega$ for every $1\le i\le s$. Thus, if
\[
\tau=\sum_{\substack{m\in \bZ^s,\, |m|=l} } z^m\tau_m
\] 
is a tau-function of the $s$-component KP hierarchy of charge $l$, then for $a=(a_1,\ldots, a_s)\in \bZ^s$, 
\[
\sum_{\substack{m\in \bZ^s,\, |m|=l} } (-1)^{\sum_{i=2}^{s} (-1)^{a_i(m_1+\cdots+m_{i-1})}}
 z^{m+a} \tau_m
\]
 is its  tau-function of charge $l+|a|$.
\end{remark}
\begin{lemma} 
   Let 
 $X^{(a)}=\sum_{k>M} A^{(a)}_k{ \psi^+}^{(a)}_k$,  where $A^{(a)}_k\in \bC$, $M\in \bZ$,  $a=1,\dots, s$. Let $X=\sum_{a=1}^s X^{(a)}$. Then
\begin{align*}
&\left(X^{(a)}\right)^2=0,\quad X^2=0,\\
&X^{(a)} X^{(b)}+X^{(b)}X^{(a)}=0 \quad \text{for}\quad a\ne b,\\
&\psi^{- (a)}_k X^{(b)}+X^{(b)}\psi^{- (a)}_k = \delta_{a,b}A^{(b)}_k.
\end{align*}
\end{lemma}
\begin{proof}
Note that $X^{(a)}$, $(X^{(a)})^2$, $X$ and $X^2$ are well-defined operators by Proposition \ref{ferm_multi} (a).
Identity $\left(X^{(a)}\right)^2=0$ follows from the KP case, Lemma \ref{commuteKPlemma}.
 Commutation relations  for $a\ne b$ follow from (\ref{mf3}), (\ref{mf4}).
Finally, 
\[
X^2=\sum_{a=1}^{s} (X^{(a)})^2+ \sum_{s\ge a>b\ge 1}( X^{(a)}X^{(b)}+ X^{(b)}X^{(a)} )=0.
\]
\end{proof}
\begin{lemma}  Let 
 $X^{(a)}=\sum_{k>M} A^{(a)}_k{ \psi^+}^{(a)}_k$  where $A^{(a)}_k\in \bC$, $M\in \bZ$, $a=1,\dots, s$. Let  $X=\sum_{a=1}^s X^{(a)}$. Then
\begin{align*}
		\Omega (X^{(a)}\otimes X^{(a)})= (X^{(a)}\otimes X^{(a)})\Omega\quad \text{and}\quad 
		\Omega (X\otimes X)= (X\otimes X)\Omega.
	\end{align*}
\end{lemma}
\begin{proof} For obvious reasons 
$\Omega^{(b)} (X^{(a)}\otimes X^{(a)})= (X^{(a)}\otimes X^{(a)})\Omega^{(b)}$ for  $a\ne b$,  
while $\Omega^{(a)} (X^{(a)}\otimes X^{(a)})= (X^{(a)}\otimes X^{(a)})\Omega^{(a)}$
from the KP case  Lemma \ref{OmXKPlemma}.
Then the statement follows. 
\end{proof}

 \begin{corollary}\label{corsKP}
 Let $\tau\in \B^{\otimes s}$  be a 
 tau-function of  charge $m$  of the  $s$-component KP  hierarchy. Let 
 $X^{(a)}=\sum_{k> M} A^{(a)}_k{ \psi^+}^{(a)}_k$  where $A^{(a)}_k\in \bC$, $M\in \bZ$, $a=1,\dots, s$.   Let  $X=\sum_{a=1}^s X^{(a)}$.
Then  $  \tau^{\prime (a)} =X^{(a)}\tau$ and $\tau^\prime =X\tau$  are   tau-functions of this hierarchy  of charge $m+1$.
 \end{corollary}
\subsection{Generating function for polynomial tau-functions of the $s$-component KP  hierarchy}\label{gen_sKP}
For $1\le a_i\le s$, $i=1,\dots, l$ let  
\begin{align*}
G^{(a_l,\dots, a_1)}(u_1,\dots, u_l)=\prod_{i<j} (u_j-u_i)^{\delta_{a_i, a_j}} \prod_{i=1}^{l}  H^{(a_i)}(u_i)
\end{align*}
be a power series in the variables $u_1,\dots, u_l$ with coefficients in $\B^{\otimes s}$.

\begin{definition}
Define the  function $\varepsilon(a_l,\dots, a_1)$ on the sets of finite sequences   of  natural numbers with values in $\pm1$ by the following recurrent formula:
Set 
\[
\varepsilon(a)=1, \quad \text{and} \quad
\varepsilon(a, a_l,\dots, a_1)= (-1)^{s(a, a_l,\dots, a_1)}\varepsilon(a_l,\dots, a_1),
\]
where  $s(a, a_l, \dots, a_1)=\#\{a_i|\, a_i<a,\, i=1,\dots, l\}$.
\end{definition}

The following lemma is straightforward. 
\begin{lemma} \label{eps}
\begin{enumerate}[label=\alph*)]
\item
\[
\varepsilon(\overbrace{s,\ldots,s}^{m_s},\ldots,\overbrace{1,\ldots,1}^{m_1})= (-1)^{\sum_{i<j}m_im_j}
\]
\item
\[
s(a, a_l,\dots, a_1) = {\sum_{r=1}^{a-1} m_r}, 
\]
where $m_r$  is the multiplicity of $r$ in the sequence $(a_l,\dots, a_1)$.
\item
Let $\rho\in S_l$ be a permutation which does not change the order of the $a_i$ which are the same. Then 
$\varepsilon(a_{\rho(l)},\ldots a_{\rho(1)})= sgn(\rho)\varepsilon(a_l,\ldots , a_1)$.
\end{enumerate}
\end{lemma}

\begin{proposition} \label{sign}
Let  $1\le a_i\le s$ , $i=1,\dots, l$. Then 
\begin{align}\label{gena}
\psi^{+(a_l)}(u_l)\dots \psi^{+(a_1)}(u_1) (1)=\varepsilon(a_l,\dots, a_1)  z_{a_1}\dots  z_{a_l} G^{(a_l,\dots, a_1)}(u_1,\dots, u_l).
\end{align}
\end{proposition}
\begin{proof}

By induction on $l$. Observe that 
$
 \psi^{+(a)}(u)(1)= H^{(a)}(u)
$.
Assume that for  $(a_1,\dots, a_l)$ formula  (\ref{gena})  holds. 
Let  $(a_l,\dots, a_1)$ be a permutation of 
\[
 \underbrace{1,\dots, 1}_{m_1}, \dots  \underbrace{s,\dots, s}_{m_s}, 
\]
where $m_i\ge 0$ and $m_1+\dots +m_s=l$.
Then from    (\ref{mP1}), Lemma \ref{eps}  b) and    Lemma \ref{propHE} we obtain
\begin{align*}
&\psi^{+(a)}(u_{l+1})\psi^{+(a_l)}(u_l)\dots \psi^{+(a_1)}(u_1)(1)= \\
&= \psi^{+(a)}(u_{l+1})\left( \varepsilon(a_l,\dots, a_1)  z_{1}^{m_1}\dots  z_{l}^{m_l}
\prod_{i<j} (u_j-u_i)^{\delta_{a_i, a_j}} \prod_{i=1}^{l}  H^{(a_i)}(u_i)\right)
\\
&=
 \varepsilon(a_l,\dots, a_1)(-1)^{\sum_{r=1}^{a-1} m_r} z_{1}^{m_1}\dots z_a^{m_a+1}  \dots   z_{l}^{m_l} u_{l+1}^{m_a}
\prod_{i=1}^{l} \left(1-\frac{u_i}{u_{l+1}}\right)^{\delta_{a_i, a}} 
\prod_{i<j} (u_j-u_i)^{\delta_{a_i, a_j}} \prod_{i=1}^{l+1}  H^{(a_i)}(u_i)
\\
&=  \varepsilon(a, a_l,\dots, a_1)z_{a_1}\dots  z_{a_{l+1}} G^{(a, a_1,\dots, a_l)}(u_1,\dots,  u_{l+1}),
\end{align*}
which proves the statement by the inductive assumption.

\end{proof}

\bigskip
 From   Proposition \ref{sign}  it follows that in the expansion 
\[
G^{(a_l,\dots, a_1)}(u_1,\dots, u_l)=\sum_{\alpha=(\alpha_1,\dots, \alpha_l)_\in\bZ_{+}^l} G_\alpha ^{(a_l,\dots, a_1)} u_1^{\alpha_1}\dots u_l^{\alpha_l}
\]
one has
\begin{align*}
G_\alpha ^{(a_l,\dots, a_1)}=\varepsilon(a_l,\dots, a_1)  z^{-1}_{a_1}\dots z^{-1}_{a_l} \psi^{+(a_l)}_{-\alpha_l-1/2} \dots \psi^{+(a_1)}_{-\alpha_1-1/2}(1).
\end{align*}

 As in Section \ref{sec2.3}, we introduce  a collection of non-zero formal  Laurent series 
  $A^{(r)}_i(u)\in \bC((u))$, $r=1,\dots, s$, $i=1,\dots, l$. Define 
the $\B^{\otimes s}$-valued  Laurent series $T^{(r)}_i(u)=A^{(r)}_i(u)H^{(r)}(u)\in \B^{\otimes s (0)}((u))$. 
Set
 \begin{align}\label{AQSKP}
 T^{(a_l,\dots, a_1)}(u_1,\dots, u_l)&=
  \prod_{j=1}^{l}A_j^{(a_j)}(u_j) G^{(a_l,\dots, a_1)} (u_1,\dots, u_l)
  \\&=
\prod_{i<j}\left(u_j-{u_i}\right)^{\delta_{a_i, a_j}}\prod_{j=1}^{l} T_j^{(a_j)}(u_j)
\in \B^{\otimes s (0)}[[u_1,\dots, u_l]][u_1^{-1},\dots, u_l^{-1}]
. \notag
\end{align}

Set also
\begin{align}\label{TSKP}
T(u_1, \dots, u_l)=\sum_{a_1,\dots, a_l=1}^{s} \varepsilon (a_l,\dots, a_1)z_{a_1}\dots z_{a_l}T^{(a_l,\dots, a_1)}(u_1,\dots, u_l),
\end{align}

Let $T^{(r)}_{i; k}$ and  $T^{(a_l,\dots, a_1)}_{\alpha}$,  $ T_{\alpha}$ be the coefficients   of the  expansions 
\begin{align}T^{(r)}_i(u)&= \sum_{k\in\bZ} T^{(r)}_{i;k} u^k,\notag
\\
 T^{(a_l,\dots, a_1)}(u_1,\dots, u_l)&= \sum_{\alpha\in \bZ^l}T^{(a_l,\dots, a_1)}_{\alpha} u_1^{\alpha_1}\dots u_l^{\alpha_l},\label{Taa}
 \\
 T(u_1,\dots, u_l)&= \sum_{\alpha\in \bZ^l}T_{\alpha} u_1^{\alpha_1}\dots u_l^{\alpha_l}.\label{TT}
\end{align}

In Section \ref{sec2.3} we expressed the generating function  of  the KP tau-functions  as   a determinant. We will show that  the generating function 
in the multicomponent KP case 
$T^{(a_l,\dots, a_1)}(u_1,\dots, u_l)$   is  a product of  appropriate determinants of the form (\ref{TKP}), First, we illustrate this  result, inspired by \cite{KL-sKP}  Theorem 4, with an example, and then  formulate and prove the general statement. We also describe polynomial tau-functions of  the bilinear  $s$-component  KP  identity as coefficients of  $T(u_1,\dots, u_l)$.
\begin{example}
We have
\begin{align*}
T^{(b,b,a,b,a)}(u_1,u_2, u_3, u_4, u_5) &=
(u_3-u_1)T^{(a)}_1(u_1)T^{(a)}_3(u_3)\\&\quad \times 
(u_5-u_4)(u_5-u_2)(u_4-u_2) T^{(b)}_2(u_2)T^{(b)}_4(u_4)T^{(b)}_5(u_5)\\
&= \det\left[ u_i^{j-1}T^{(b)}_i(u_i)\right]_{ \substack{{i=2,4,5}\\{ j=1,2,3}}}\det\left[ u_i^{j-1}T^{(a)}_i(u_i)\right]_{\substack{{i=1,3}\\{ j=1,2}}}.
\end{align*}
\end{example}
More generally,  for any choice of parameters $(a_l,\dots, a_1)$, $a_i\in\{1,\dots, s\}$,  define a partition of the set $\{1,\dots,l\}= I_1\sqcup\dots \sqcup I_s$ into the subsets $I_1,\dots, I_s$, where 
$I_r=\{i\in\{1,\dots, l\}|\, a_i=r\}$.   If  $|I_r|=m_r$, 
then $m_1+\dots+m_r=l$.
 For example, for the choice $(a_4, a_3, a_2, a_1)= (1, 3,3, 5)$, one gets $I_1=\{4\}$, $I_3=\{2,3\}$, $I_5=\{1\}$, and the rest of $I_r$'s are empty. 
\bigskip

  \begin{theorem}  \label{thmsKP}
\begin{enumerate}[label=\alph*)]
\item
For any choice of   $a_i\in\{1,\dots, s\}$, $i=1,\dots, l$,  we have
\[
T^{(a_l,\dots, a_1)}(u_1,\dots, u_l)
=\prod_{r=1}^{s}\det\left[ u_i^{j-1}T^{(r)}_i(u_i)\right]_{i\in I_r,\, j=1,\dots, m_r }.
\]
\item 
For any  $(\alpha_1,\dots, \alpha_l)\in \bZ^l$ the coefficient of the monomial $u_1^{\alpha_1}\dots u_l^{\alpha_l}$ in (\ref{AQSKP})
 is given by 
 \begin{align}\label{Tacoef}
T^{(a_l,\dots, a_1)}_\alpha
=\prod_{r=1}^{s}\det\left[ T^{(r)}_{i; \alpha_i +1 -j}\right]_{i\in I_r,\, j=1,\dots, m_r }.
\end{align}
\item 
  For any  $(\alpha_1,\dots, \alpha_l)\in \bZ^l$ the coefficient  $T^{(a_l,\dots, a_1)}_\alpha$   of the monomial $u_1^{\alpha_1}\dots u_l^{\alpha_l}$    in (\ref{Taa})  is 
 a  tau-function of  the  $s$-component KP   hierarchy.  If $A^{(r)}_{i}(u)\in \bC[u,u^{-1}]$ are Laurent polynomials for all  $i=1,\dots, l$,  $r=1,\dots, s$, then  the element $T^{(a_l,\dots, a_1)}_\alpha$ is a   polynomial tau-function.

  \item 
  For any  $(\alpha_1,\dots, \alpha_l)\in \bZ^l$ the coefficient $T_\alpha$    of the monomial $u_1^{\alpha_1}\dots u_l^{\alpha_l}$   (\ref {TT})  
is a   tau-function of  the  $s$-component KP   hierarchy.  If $A^{(r)}_{i}(u)\in \bC[u,u^{-1}]$ are Laurent polynomials for all  $i=1,\dots, l$,  $r=1,\dots, s$, then  the element
  $T_\alpha$   is a   polynomial tau-function.

\item 
Let $\tau$ be a  polynomial $\tau$-function of the  $s$-component  KP  hierarchy. Then there exists   
a collection  of  complex-valued Laurent polynomials  $A^{(r)}_i(u)\in\bC[u,u^{-1}]$
$i=1,\dots,l$, $r=1,\dots, s$,  such that
$\tau$  is a  zero-mode of  the Laurent series expansion of $T(u_1,\dots, u_l)$ in  (\ref {TSKP}).
\end{enumerate}
\end{theorem}
\begin{proof}
a)
Observe that 
\begin{align*}
\prod_{i<j}\left(u_j-{u_i}\right)^{\delta_{a_i, a_j}}= \prod_{r=1}^{s}\prod_{\substack{{i<j,}\\{ i, j\in I_r}}}\left(u_j-{u_i}\right)=\prod_{r=1}^{s}\det\left[ u_i^{j-1}\right]_{i\in I_r, j=1,\dots, m_r }
\end{align*}
and
\begin{align*}
\prod_{j=1}^{l} T_j^{(a_j)}(u_j) = \prod_{r=1}^{s}\prod_{i\in I_r} T_i^{(r)}(u_i).
\end{align*}
Then 
\[
T^{(a_l,\dots, a_1)}(u_1,\dots, u_l)=\prod_{r=1}^{s}
\left(\det\left[ u_i^{j-1}\right]_{i\in I_r, j=1,\dots, m_r } \prod_{i\in I_r} T_i^{(r)}(u_i) \right)=
  \prod_{r=1}^{s}\det\left[ u_i^{j-1}T_i^{(r)}(u_i)\right]_{\substack{i\in I_r,\\ j=1,\dots, m_r }}.
\]
b) This follows from a) and (\ref{TcKP}).
\\
c)  Let $A^{(r)}_j(u)=\sum_{k\ge M_{j}^{(r)}} A^{(r)}_{j,k-1/2} u^k $, where $A^{(r)}_{j, k-1/2}\in \bC$, $M_j^{(r)}\in \bZ$,  $r=1,\dots, s$, $j=1,\dots, l$,  $k\in\bZ$.
By Propostion \ref{gena} and definition (\ref{AQSKP}),
\begin{align*}
A^{(a_l)}_l(u_l) \dots A^{(a_1)}_1(u_1)&\psi^{+(a_l)}(u_l) \dots \psi^{+(a_1)}(u_1)\, (1)\\
&=  \varepsilon(a_l,\dots, a_1) z_{a_1}\dots z_{a_l}\,T^{(a_l,\dots, a_1)}(u_1, \dots, u_l). 
\end{align*}
Then the coefficient of $u_1^{\alpha_1}\dots  u_l^{\alpha_l}$  in $T^{(a_l,\dots, a_1)}(u_1, \dots, u_l)$   can be represented in  the form
$
 T^{(a_l,\dots, a_1)}_\alpha =  \varepsilon(a_l,\dots, a_1)z^{-1}_{a_1}\dots z^{-1}_{a_l}X^{(a_l)}_l\dots X^{(a_1)}_1\, (1)
$
with 
\begin{align}\label{Xmulti}
X^{(a)}_j=\sum_{{ i\ge M^{(a)}_j-\alpha_j-1/2} }A^{(a)}_{j,\, \alpha_j+i }\psi^{+(a)}_{i}.
\end{align}
By Corollary \ref{corsKP},  it  is 
a tau-function of the   $s$-component KP hierarchy.

 d) Similarly,  the coefficient of 
$u_1^{\alpha_1}\dots  u_l^{\alpha_l}$ in $ T(u_1,\dots, u_l)$  is of the form
\begin{align*}
\sum_{a_1,\dots, a_l =1}^{s}  \varepsilon(a_1,\dots, a_l)z_{a_1}\dots z_{a_l}T^{(a_l,\dots, a_1)}_\alpha&=
\sum_{a_1,\dots, a_l =1}^{s}X^{(a_l)}_l\dots X^{(a_1)}_1\, (1)
=X_l\dots X_1\, (1),
\end{align*}
where
$X_j=\sum_{a=1}^{s}X^{(a)}_j. 
$
By Corollary \ref{corsKP}, this coefficient is  a tau-function of the  $s$-component  KP hierarchy.

 If  all of  $A^{(r)}_{i}(u)$ are Laurent polynomials, the sum (\ref{Xmulti})  has finitely many non-zero terms, and the coefficient of $u_1^{\alpha_1}\dots  u_l^{\alpha_l}$ is a finite linear combination of  elements of the form $\psi^{+(a_1)}_{k_1}\dots \psi^{+(a_l)}_{k_l}(1)$, hence it is a polynomial $\tau$-function.

e) Following   \cite{KL-sKP}, \cite{KL-2003},
by the boson-fermion correspondence we identify  $\B^{\otimes s}$  with the semi-infinite wedge space $\Lambda^{\frac{1}{2}}(\bC^\infty)$,   where  $\bC^{\infty}=\oplus_{a=1}^{s}\oplus_{j\in \bZ}\bC e_j^{(a)}$.
As  argued in   \cite{KL-sKP}, without loss of generality we can    consider polynomial   $\tau$-functions  of  the form
 \begin{align}\label{tau0}
\tau= f_l\wedge \dots \wedge f_1\wedge |0\rangle, 
 \end{align}
where
\[
f_j=\sum_{a=1}^{s} f_j^{(a)},\quad \quad  f_j^{(a)}=\sum_{k=1}^{M_j^{(a)}} b_{kj}^{(a)}e_k^{(a)}, \quad j=1,\dots, m,\quad  a=1,\dots, s,
\]
and  \begin{align*}
 |0\rangle=&
(e^{(1)}_0\wedge e^{(1)}_{-1}\wedge e^{(1)}_{-2}\wedge\dots)\wedge\dots \wedge 
(e^{(s)}_0\wedge e^{(s)}_{-1}\wedge e^{(s)}_{-2}\wedge\dots),
 \end{align*}
  since under the boson-fermion correspondence  the image of  any other $\tau$-function in $\B^{\otimes s}$ is a $z^a$-multiple of the image of a function of the form  (\ref{tau0}): 
 $\tau^\prime=z^a\tau
 $.

Using that $\psi^{+(a)}_{-k+1/2} \left(e^{(c_1)}_{i_1}\wedge e^{(c_2)}_{i_2}\dots\right)=\ e^{(a)}_k\wedge e^{(c_1)}_{i_1}\wedge e^{(c_2)}_{i_2}\dots$ 
we can write 
 \begin{align*}
 f_l\wedge \dots \wedge f_1\wedge |0\rangle
 =\sum_{a_1,\dots, a_l=1}^sY^{(a_l)}_l\dots Y^{(a_1)}_1 |0\rangle
 =Y_l\dots Y_1|0\rangle,
 \end{align*}
 where
 \begin{align}\label{Xb}
Y_j^{(a)}=\sum_{-M_j^{(a)}+1/2\le k\le-1/2} b_{-k+1/2,j}^{(a)}\psi_{k}^{+(a)}, \quad j=1,\dots, l,
\quad \text{and}\quad 
Y_j=\sum_{a=1}^{s}Y_j^{(a)}.
\end{align}
Set 
\begin{align}\label{choiceA}
A^{(a)}_j(u)=u^{\alpha_j}\sum_{-M_j^{(a)}+1/2\le k\le -1/2} b_{-k+1/2,j}^{(a)}u^{k-1/2}\quad  a=1,\dots, s,
\end{align}

Consider $T^{(a_l,\dots, a_1)}(u_1,\dots, u_l)$ defined with the choice (\ref{choiceA}) of Laurent series multipliers. 
Then  $X^{(a)}_j$  in (\ref{Xmulti})  coincides with  $Y^{(a)}_j$ 
 in  (\ref{Xb}).

  Hence,  the coefficient of  $u_1^{\alpha_1}\dots u_l^{\alpha_l}$
 in   $T^{(a_1,\dots, a_l)}(u_1,\dots, u_l)$  
 coincides with  \[\varepsilon(a_l,\dots, a_1) z^{-1}_{a_1}\dots  z^{-1}_{a_l}Y_l^{(a_l)}\dots Y_1^{(a_1)}|0\rangle=\varepsilon(a_l,\dots, a_1) z^{-1}_{a_1}\dots  z^{-1}_{a_l} f^{(a_l)}_l\wedge\dots \wedge f^{(a_1)}_1\wedge |0\rangle,\]
  while the coefficient of $u_1^{\alpha_1}\dots u_l^{\alpha_l}$ in (\ref {TSKP})  
 is
 $\sum_{a}f^{(a_l)}_l\wedge\dots \wedge f^{(a_1)}_1\wedge |0\rangle =f_l\wedge \dots \wedge f_1\wedge |0\rangle=\tau$. 
 In particular,  consider
 $\alpha_1=\dots =\alpha_l=0$ to realize  $\tau$ as a zero-mode of  $T(u_1,\dots, u_l) $.
  \end{proof}


\begin{remark} \label{rem_52}
Consider $s$  copies of the KP-hierarchy and let  $\tau^{(1)}, \dots, \tau^{(s)}$  be their   tau-functions. 
Then the product $\tau^{(1)} \dots \tau^{(s)}$ is   a tau-function of the $s$-component KP hierarchy. This follows from the definition of the 
tau-function of the $s$-component KP hierarchy. 
\end{remark}

We now want to generalize Theorem \ref{thmsKP} b) and c). Let  $\tau=\sum_{\substack{m\in \bZ^s,|m|=l} } z^m\tau_m\in  \B^{\otimes s}\, ^{(l)}$ be 
tau-function of charge $l$ of the $s$-component KP  hierarchy.
Then $\tau_m$ is the coefficient of  the monomial
$u_1^{\alpha_1}u_2^{\alpha_2}\cdots u_l^{\alpha_l}z_1^{m_1}z_2^{m_2}\cdots z_s^{m_s}$ in $T(u_1,u_2,\ldots,u_l)$. To obtain this coefficient, one has to sum all $\varepsilon(a_l,\ldots a_1)T^{(a_l,\ldots,a_1)}(u_1,\ldots u_1)$, where $(a_1,\dots,a_l)$ is  a permutation of 
$(\overbrace{1,\ldots,1}^{m_1 \text{ times}},\ldots, \overbrace{s,\ldots,s}^{m_s \text{ times}})$.
Using Theorem \ref{thmsKP} a), we find that
\[
T^{(s,\ldots,s\ldots,1\ldots,1)}(u_1,\dots, u_l)
=\prod_{r=1}^{s}\det\left[ u_i^{j-1}T^{(r)}_i(u_i)\right]_{m_1+\cdots + m_{r-1}< i \le m_1+\cdots+m_r,\, j=1,\dots, m_r }.
\]
We want to write this as one determinant. Introduce, for $m\in \mathbb{Z}^s$, the  function $\sigma_m(x)$ on the interval $[0,l]$,
which is  piece-wise linear of slope $45^\circ$, continuous at all points except for points $\{m_1, m_1+m_2, \dots, m_1+\dots+ m_{s-1}\}$, and\begin{equation}
\label{sigma}
\sigma_m(0)=0, \quad \sigma_m(m_1+\cdots+m_{r})=m_r\quad\text{for}\quad  r=1, \dots, s.
\end{equation}
Consider the subgroup $S_l'=S_{m_1}\times S_{m_2}\times\cdots \times S_{m_s}$ of $S_l$.
Using Lemma \ref{eps} a) and  the notation 
\begin{align}\label{520a}
(a_l,\ldots,a_1)=(\overbrace{s,\ldots,s}^{m_s \text{ times}},\ldots, \overbrace{1,\ldots,1}^{m_1 \text{ times}})
\end{align}
 we obtain
\[
\varepsilon(a_l,\ldots a_1)T^{(a_l,\ldots,a_1)}(u_1,\dots, u_l)=(-1)^{\sum_{i<j}m_im_j}
\sum_{\pi\in S_l'}sgn(\pi) \prod_{i=1}^l u_{\pi(i)}^{\sigma_m(i)-1} T_{\pi(i)}^{(a_i)}(u_{\pi(i)}) .
\]

Next, let $\rho\in S_l$ be a permutation which does not change the order of the $a_i$ which are the same. Then,  by Lemma \ref{eps}\,c),
$\varepsilon(a_{\rho(l)},\ldots a_{\rho(1)})= sgn(\rho)\varepsilon(a_l,\ldots , a_1)$, and $T^{(a_{\rho(l)},\ldots,a_{\rho(1)})}(u_1,\dots, u_l))=T^{(a_l,\ldots,a_1)}(u_{\rho(1)},\dots, u_{\rho(l)})$.
 
 Thus 
\begin{align*}
\varepsilon&(a_l,\ldots , a_1)
\varepsilon(a_{\rho(l)},\ldots a_{\rho(1)})T^{(a_{\rho(l)},\ldots,a_{\rho(1)})}(u_1,\dots, u_l)=
 sgn(\rho)T^{(a_l,\ldots,a_1)}(u_{\rho(1)},\dots, u_{\rho(l)})\\
=&sgn(\rho)
\sum_{\rho\pi | \pi\in S_l'}sgn(\pi) \prod_{i=1}^l
 u_{ \rho\pi(i)   
}^{\sigma_m(i)-1}
  T_{\rho\pi(i)  
}^{(a_i)}
  (u_{\rho\pi(i) 
}).
\end{align*}
Hence to obtain the coefficient of 
$z_1^{m_1}\cdots z_s^{m_s}$ of 
$\varepsilon(a_l,\ldots , a_1)T(u_1,\ldots, u_s)$, we have to sum over all such permutations $\rho$ that do not change the order of the $a_i$ that are the same.
Therefore  this coefficient  is equal to 
\[
\sum_{\rho}
\sum_{\rho\pi |  \pi\in S_l'}sgn(\rho\pi) \prod_{i=1}^l u_{ \rho\pi(i)}^{\sigma_m(i)-1} T_{\rho\pi(i)  
}^{(a_i)}(u_{\rho\pi(i)}).
\]
This means that we sum over all permutations of $S_l$ and hence that the  coefficient of 
$z_1^{m_1}\cdots z_s^{m_s}$ of 
$\varepsilon(a_l,\ldots , a_1)T(u_1,\ldots, u_s)$ is equal to
\begin{align*}
\sum_{\pi\in S_l}
sgn(\pi) \prod_{i=1}^l u_{\pi(i)}^{\sigma_m(i)-1} T_{\pi(i)}^{(a_i)}(u_{\pi(i)})
=\det\left[ u_{i}^{\sigma_m(j)-1} T_i^{(a_j)}(u_i)\right]_{i,j=1,\ldots, l } .
\end{align*}
Thus we have obtained the following
 \begin{theorem}  \label{thmsKP2}
For any  $(\alpha_1,\dots, \alpha_l)\in \bZ^l$ the coefficient of the monomial $u_1^{\alpha_1}\dots u_l^{\alpha_l}$ in (\ref{TSKP}) is a tau-function
$\tau=\sum_{\substack{m\in \bZ^s,|m|=l} } z^m\tau_m$  of charge $l$ of the $s$-component KP hierarchy. In particular, $\tau_m$ is the  coefficient of $u_1^{\alpha_1}\dots u_l^{\alpha_l}$ of
\begin{equation}
\label{taum}
T_m(u_1,\ldots,u_l)= (-1)^{\sum_{i<j}m_im_j}
\det\left[ u_{i}^{\sigma_m(j)-1} T_i^{(a_j)}(u_i)\right]_{i,j=1,\ldots, l } ,
\end{equation}
where  we use notation (\ref{520a}).
\end{theorem}

\subsection{Corollary.}\label{csKP}
	 Consider non-zero Laurent  polynomials  $A^{(a)}_j(u_j)$ in the definition of $T(u_1,\dots, u_l)$ in  (\ref{AQSKP}).
One can represent each of  them in the form 
	\[
	A_j^{(a)}(u)= u^{M_j^{(a)}} b_{j}^{(a)}\sum_{i\in \bZ_+} a_{j, i}^{(a)} u^i, \quad 1\le j\le l,\ 1\le a\le s,
	\]
for suitable choices of $M^{(a)}_j\in \bZ$, $b_j^{(a)}, a_{j,i}^{(a)}\in \bC$,   $a_{j,0}^{(a)}=1$, where  all  $b_j^{(a)}$ are non-zero.
Then using analogous arguments as in Section \ref{rcKP}, one can deduce that there exist constants  constants $c_{j,i}^{(a)}\in \bC$, such that 
\[ 
T_j^{(a)}(u)=u^{M_j^{(a)}}b_j^{(a)}\sum_{k=0}^\infty 
S_k(t_1^{(a)}+c_{j,1}^{(a)},\,t_2^{(a)}+c_{j,2}^{(a)},\dots) u^k.
\]
Hence,
$
T_{j;k}^{(a)}=b_j^{(a)}S_{k-M_j^{(a)}}(t_1^{(a)}+c_{i,1}^{(a)},\,t_2^{(a)}+c_{i,2}^{(a)},\dots),
$
and by \eqref{tau0} and    Theorem \ref{thmsKP2},  we obtain a polynomial  tau-function of the  $s$-component  KP  hierarch of charge $l\ge 0$, such that, in notation  (\ref{520a}), it has the form
\begin{equation}
\label{tau-expl}
\sum_{\substack{m\in \bZ_+^s,|m|=l} } 
 (-1)^{\sum_{i<j}m_im_j}z^m
\det\left[b_i^{(a_j)} S_{\alpha_i-M_i^{(a_j)}+1-\sigma_m(j)}(t_1^{(a_j)}+c_{i,1}^{(a_j)},\,t_2^{(a_j)}+c_{i,2}^{(a_j)},\dots)\right]_{i,j=1,\dots,l}.
\end{equation}
\begin{remark}
\label{Remm}
Note first,  that the sum in \eqref{tau-expl}  is finite, since there are only finitely many $m\in \mathbb{Z}_+^s$ such that $|m|=m_1+m_2+\cdots+m_s =l$. Note secondly, that any polynomial tau-function corresponds  an element of charge 0 of the form 
 \begin{align}\label{tau000}
\tau= g_l\wedge \dots \wedge g_1\wedge e_{-l}\wedge e_{-l-1}\wedge e_{-l-2}\wedge \cdots,
 \end{align}
for certain $l$. This gives, when translating the charge by $l$ an element of the form \eqref{tau0}.
 In this way we obtain that any polynomial tau-function of the $s$-component KP hierarchy is of the form \eqref{tau-expl}, when translated  over the lattice $\bZ^s$ (see also Remark \ref{rem_22}).
\end{remark}
\section{Polynomial $\tau$-functions of the bilinear $\lambda$-KdV hierarchy}
\label{SecKdV}
In this section we consider the $\lambda$-KdV hierarchy for  the partition $\lambda=(\lambda_1,\lambda_2, \ldots ,\lambda_s)$
consisting  of $s$ positive  parts, which is a  reduction of the $s$-component KP hierarchy introduced in \cite{KL-2003}. We again use the ($s$-component) boson-fermion correspondence.

\subsection{Bilinear equations for the $\lambda$-KdV hierarchy }

The  $\lambda=(\lambda_1,\lambda_2, \ldots ,\lambda_s)$-KdV hierarchy is a reduction of the $s$-component KP hierarchy, for which, 
in addition to equation (\ref{multibKP}),
the tau-function satisfies the following  equations 
\begin{align}\label{multibKPj}
\Omega_j ( \tau\otimes \tau)= 0, \qquad j\in \bN,
\end{align} where
\begin{align*}
\Omega_j=\sum_{a=1}^{s}\Omega^{(a)}_j,\quad 
\Omega^{(a)}_j=\sum_{k\in \bZ+\frac{1}{2}} \psi_k^{+(a)}\otimes \psi^{-(a)}_{j\lambda_a-k}.
\end{align*}
\subsection{Polynomial tau-functions for the  $\lambda$-KdV hierarchy }
See \cite{KL-sKP} for more details. 
For polynomial $s$-component tau-functions the additional equations (\ref{multibKPj}) mean that
\[
\sum_{a=1}^s \frac{\partial \tau}{\partial t^{(a)}_{j\lambda_a}}=0, \qquad j\in \bN,
\]
or, in the fermionic picture,  that for an element (\ref{tau0}),  one has  $\Lambda f_j\wedge \tau=0$, for $j=1,\ldots, l$, where the operator 
${\Lambda}$ on $\bC^\infty$ is defined  by 
 \begin{align}\label{Lambda}
\Lambda e_k^{(a)}= e_{k-\lambda_a}^{(a)},\quad 1\le a\le s,\ k\in\mathbb{Z}.
 \end{align}
This means that  (\ref{tau0}) depends only  on $r<|\lambda|$  elements  $f_j\in \bC^\infty$, namely that 
 \begin{align}\label{taur}
\tau=f_1\wedge \Lambda f_1\wedge\cdots\wedge   \Lambda^{k_1-1}f_1\wedge  f_2\wedge \cdots\wedge  \Lambda^{k_2-1} f_2\wedge \cdots\wedge   \Lambda^{k_r-1} f_r\wedge |0\rangle \ \text{and }\Lambda^{k_j}f_j\wedge |0\rangle =0,
 \end{align}
 where  $j=1,\ldots, r<|\lambda|$ and   $k_1+k_2+\cdots+ k_r=l$.
Asuming that $f_j$ has the form 
\[
f_j=\sum_{a=1}^{s} f_j^{(a)},\quad \quad  f_j^{(a)}=\sum_{k=1}^{N_j^{(a)}} b_{kj}^{(a)}e_k^{(a)}, \quad j=1,\dots, r,\quad  a=1,\dots, s,
\]
we obtain, since $\Lambda^{k_j}f_j\wedge |0\rangle =0$, that 
\[
k_j\lambda_a\ge N_j^{(a)},\quad\text{for all }1\le j\le r, \ 1\le a\le s.
\]
Since we assume that $\Lambda^{k_j-1}f_j\wedge |0\rangle \ne 0$, we find that 
\begin{equation}\label{ki}
k_j= \max \{ \lceil \frac{ N_j^{(a)}}{\lambda_a}\rceil\ | \, 1\le a\le s \},  \quad   1\le j\le r.
\end{equation}
An element of the form (\ref{taur}) can be 
 achieved by assuming that
 $\tau$ is the zero-mode of the Laurent series  expansion of $T(u_1,\ldots,u_l)$, where   $T(u_1,\ldots,u_l)$  depends only on $rs$  Laurent polynomials $ B_i^{(a)}(u)$, with $1\le i\le r$ and  $1\le a\le s$, 
such that  $f_j$ can be replaced by the the zero-mode of 
\[
\sum_{a=1}^s B_i^{(a)}(u)\sum_{\ell\in\bN}  \psi^{+(a)}_{\frac12-\ell} u^\ell=\sum_{a=1}^s \sum_{k\in \bZ}\sum_{\ell\in\bN}B_{i,k}^{(a)} \psi^{+(a)}_{\frac12-\ell} u^{k+\ell} .
\]
Hence,  we can assume that $B_{i,k}^{(a)}=0$ for $k<-N_j^{(a)}$, in other words
\begin{equation}
\label{Bj}
 B_i^{(a)}(u)=u^{-N_i^{(a)}}b_{i}^{(a)}\sum_{\ell \in \bZ_+} a_{i, \ell}^{(a)} u^\ell.
\end{equation}
Then the Laurent polynomials  $A_j^{(a)}(u)$, for $1\le a\le s$ and $1\le j\le l$, which define $T(u_1,\ldots, u_l)$ by (\ref{AQSKP}),  (\ref{TSKP}), are given by 
\begin{align*}
&A_1^{(a)}(u)=B_1^{(a)}(u),A_2^{(a)}(u)=u^{\lambda_a}B_1^{(a)}(u), \ldots ,A_{k_1-1}^{(a)}(u)=u^{(k_1-1)\lambda_a}B_1^{(a)}(u),\\
&A_{k_1}^{(a)}(u)=B_2^{(a)}(u),\ldots
  ,A_{k_1+k_2}^{(a)}(u)=u^{(k_2-1)\lambda_a}B_2^{(a)}(u),\\
&A_{k_1+k_2+1}^{(a)}(u)=B_3^{(a)}(u),\ldots\ldots ,
 A_{l}^{(a)}(u)=u^{(k_r-1)\lambda_a}B_r^{(a)}(u),
\end{align*}
Recall  the  function $\sigma_k(x)$ for $k\in \bZ^r$ on the interval $[0,l]$, introduced in Section \ref{gen_sKP}. In addition, consider 
 the step function  $\gamma_k(x)$ on the interval $[1,l]$
which is  piece-wise constant, continuous at all points except for points $\{ k_1, k_1+k_2, \dots,k_1+\dots+ k_{r-1}\}$, and
\begin{equation}
\label{gamma}
\gamma_k(1)=1, \ 
 \gamma_k(k_1+\cdots+k_{j})= j, \ \text{for }j =1, \dots, r.
\end{equation} 
Then 
\[
A_j^{(a)}(u)=u^{(\sigma_k(j)-1)\lambda_a}B^{(a)}_{\gamma_k(j)}(u),\ \text{where  the}\ B_i^{(a)}(u)\ \text{are as in}\ (\ref{Bj}).
\]
Then (see also Remark \ref{Remm} and \ref{rem_22})
by   formula (\ref{tau-expl}), where we choose all $\alpha_i=0$ and $M_i^{(a)}=(\sigma_k(i)-1)\lambda_{a}-N_{\gamma_k(i)}^{(a)}$,  any  polynomial  tau-function  of the $\lambda$-KdV  hierarchy can be translated over the lattice $\bZ^s$,  such that it becomes a tau function of charge $l\ge 0$ for certain $l$ of the $\lambda$-KdV  hierarchy  of the form, in notation (\ref{520a}),
\begin{align}
\label{tau-expl2}
\sum_{\substack{m\in \bZ_+^s,\\|m|=l} } &
 (-1)^{\sum_{i<j}m_im_j}z^m \notag
 \\ \times
&\det\left[b_{\gamma _k(i)}^{(a_j)} S_{N_{\gamma_k(i)}^{(a_j)}-(\sigma_k(i)-1)\lambda_{a_j}+1-\sigma_m(j)}(t_1^{(a_j)}+c_{\gamma_k(i),1}^{(a_j)},\,t_2^{(a_j)}+c_{\gamma_k(i),2}^{(a_j)},\dots)\right]_{i,j=1,\dots,l},
\end{align}
where the $k_i$'s are given by (\ref{ki}).

  \end{document}